%% file: main.tex
\documentclass[11pt,reqno]{amsart}

\usepackage{
  amssymb,
  cancel, 
  float, 
  graphicx, 
  hyperref,
  mathabx
}

\usepackage[normalem]{ulem}

\usepackage[bottom=3.5cm,left=2.6cm,right=2.6cm,top=3cm]{geometry}

\newtheorem{theorem}{Theorem}
\newtheorem{corollary}[theorem]{Corollary}
\newtheorem{lemma}[theorem]{Lemma}
\newtheorem{proposition}[theorem]{Proposition}
\numberwithin{theorem}{section}

\theoremstyle{definition}
\newtheorem{conjecture}[theorem]{Conjecture}
\newtheorem{definition}[theorem]{Definition}

\newtheorem{remark}[theorem]{Remark}

\def\red{\color{red}} 
\def\blue{\color{blue}}

\newcommand{\CC}{\mathbb{C}}

\newcommand{\Pa}{\operatorname{P}}
\newcommand{\PPa}{\operatorname{PP}}

\newcommand{\Cfr}{\mathfrak{C}}
\newcommand{\Jfr}{\mathfrak{J}}

\newcommand{\Sfr}{\mathfrak{S}}
\newcommand{\Xfr}{\mathfrak{X}}

\newcommand{\Brfr}{\mathfrak{Br}}

\newcommand{\RC}{\mathcal{R}\mathfrak{C}}
\newcommand{\dbrack}[1]{\ldbrack #1\rdbrack}
\newcommand{\dbrace}[1]{\{\kern-0.25em\{#1\}\kern-0.25em\}}

\newcommand{\D}{\mathcal{D}}
\newcommand{\E}{\mathcal{E}}
\newcommand{\G}{\mathcal{G}}
\newcommand{\J}{\mathcal{J}}
\renewcommand{\L}{\mathcal{L}}
\newcommand{\R}{\mathcal{R}}

\newcommand{\BP}{\mathcal{BP}}
\newcommand{\EP}{\mathcal{EP}}

\renewcommand{\H}{\mathcal{H}}
\renewcommand{\P}{\mathcal{P}}

\newcommand{\wF}{\widetilde{F}}
\newcommand{\wG}{\widetilde{G}}

\newcommand{\Par}{\operatorname{Par}}

\newcommand{\End}{\operatorname{End}}

\input{pics/pics.tex}

\begin{document}

\title{Party-Hecke algebras}

\author[Arcis]{Diego Arcis}
\address{Departamento de Matem\'aticas, Universidad de La Serena, Cisternas 1200, 1700000 La Serena, Chile}
\email{diego.arcis@userena.cl}

\author[Juyumaya]{Jes\'us Juyumaya}
\address{Instituto de Matem\'aticas, Universidad de Valpara\'iso, Gran Bretaña 1111, 2340000 Valparaíso, Chile}
\email{juyumaya@gmail.com}

\date{}\keywords{}
\thanks{The first named author acknowledges the financial support of DIDULS/ULS, through
the project PR2553853.}
\setcounter{tocdepth}{1}

\begin{abstract}
Party-Hecke algebras are introduced as a two-parameter deformation of party algebras, where one parameter deforms the  party generators and the other deforms the elementary transpositions. We construct a basis for this algebra and show that it can be realized as a quotient of the algebra of braids and ties. Furthermore, we study the party monoid and its relationship with the tied symmetric monoid and their associated algebras.
\end{abstract}

\maketitle
\tableofcontents

\section{Introduction}

Deformation algebras are central to the classification of algebraic structures, as well as to quantum physics and knot theory, among other fields. One class of deformations consists of those obtained by deforming the centralizer algebras of representations, such as the Hecke and BMW algebras, from which the HOMFLYPT and Kauffman polynomials are derived, respectively. These centralizers are cornerstones of the so-called Schur--Weyl duality, the framework from which the deformed algebra studied here also arises. This duality is a central topic in invariant theory and originates from the duality between the general linear group $G=\operatorname{GL}_n(\CC)$ and the symmetric group $\Sfr_d$ acting on the tensor space $V^{\otimes d}$, where $V=\CC^n$. In this setting, the former acts diagonally and the latter by permuting the tensor factors. These actions commute, and the Schur--Weyl duality theorem (see, for instance, \cite[Theorem 3.4]{CoPr2017}) establishes deep connections between the irreducible representations of $G$ and the representation theory of the symmetric group. In particular, for $n\geq d$, the centralizer $\End_G(V^{\otimes d})$ is the group algebra $\CC[\Sfr_d]$. Later, in \cite{Brauer1937}, Brauer studied the analogous situation where $G$ is replaced by the orthogonal subgroup $O_n$. Specifically, when $n\geq d$, the centralizer $\End_{O_n}(V^{\otimes d})$, denoted by $\Brfr_n(d)$ and called the Brauer algebra, is completely described by transpositions and certain elements which, in diagrammatic language are known now as tangles. In \cite{Jones94}, Jones studied the centralizer $\End_{\Sfr_n}(V^{\otimes d})$, where $\Sfr_n$ is regarded as a subgroup of $G$, and introduced the partition algebra. This algebra was also introduced independently by Martin in \cite{Martin1994}. Regarding $\Sfr_n$ as the complex reflection group $G(1,1,n)$, it is natural to investigate the centralizers of the complex reflection group $G(d,r,n)$. Tanabe studied the centralizer and provided a presentation for it, in which the generators consist of the elementary transpositions and certain party elements (see \cite[Theorem 3.1]{Ta1997}). Subsequently, building on the aforementioned work by Tanabe, Kosuda introduced, for any $\delta\in\CC^\times$, the party algebra $\P_k(\delta)$ in \cite[Definition 1.1]{KoOJM2006}, defining it through a presentation by generators and relations. Moreover, he proved that $\P_k(\delta)$ can be realized as the centralizer $\End_{G(d,1,n)}(V^{\otimes k})$ for $n\geq k$ and $d>k$, see \cite[Propositions 1.2 and 1.3]{KoOJM2006}; in fact, $\P_k{(1)}=\CC[\P_k]$, where $\P_k$ denotes the party monoid, also known as the monoid of uniform block transpositions. Recall that deformations of $\CC[\Sfr_n]$ and $\Brfr_n(d)$ yield the Hecke and BMW algebras, respectively. The key point of these deformations is that the elementary transpositions are deformed in both algebras. For instance, in the Hecke algebra setting, the relation $s^2=1$ is $u$-deformed into $s^2=u+(u-1)s$, leading to non-trivial braid group representations used in the construction of knot invariants. This article introduces and studies a two-parameter deformation of the party algebra, which we call the \emph{Party-Hecke algebra}. In short, the Party-Hecke algebra is a two-parameter deformation in which one parameter deforms the party generators while the other deforms the elementary permutation generators. In what follows, we outline the structure and main results of the article.

Section~\ref{SecPreliminaries} provides the tools used in this paper. In particular, Subsections~\ref{SecSetPartitions}--\ref{SubSecSymmetricBraidGroups} cover the standard background on the monoids used here, namely: the monoid of set partitions $P_n$, the partition monoid $\Cfr_n$, the symmetric group $\Sfr_n$, and the braid group $B_n$. In Subsection~\ref{SubSecRamifiedMon}, we provide the definition of ramified monoids \cite[Section 4]{AiArJu23} \cite{AiArJu24}, and we recall the usual presentation of the ramified monoid $T\Sfr_n=P_n\rtimes\Sfr_n$ of the group $\Sfr_n$ (Subsubsection~\ref{RamifiedSymmetricMonoid}). Subsequently, we recall the definition of the so-called tied braid monoid $TB_n$ (Subsubsection~\ref{SubSubSecTiedBraidMon}) \cite[Section 3]{AiJu16}. Finally, we recall the definition of the Brauer monoid $\Brfr_n$ and its ramified version $\R\Brfr_n$ (Subsubsections~\ref{BrauerSSS}--\ref{SubSubSec244}). Subsection~\ref{SubSecTwiMonAlg} focuses on twisted monoid algebras \cite[Section 3]{Wilcox2007}, noting that the partition algebra can be regarded as a twisted monoid algebra \cite[Section 7]{Wilcox2007}. Finally, in Subsection~\ref{SecIwahoriHeckeBT}, we recall the definition of the two-parameter algebra of braids and ties \cite{AiJu20}.

Section \ref{SecParMonAlg} begins by recalling the definition of the party monoid $\P_n$ (Subsubsection~\ref{subPartyMon}), and in particular, we establish two key results: Proposition~\ref{TSn/R}, which characterizes $\P_n$ as a quotient of the tied symmetric monoid $T\Sfr_n$, and Proposition~\ref{partyNF}, which provides a normal form for the elements of $\P_n$. In Subsubsection~\ref{SubSubSec312}, we recall the definition of the party algebra $\P_n(\delta)$, where $\delta\in\CC^\times$ (\cite[Theorem 1.1]{Kos00}). Next, in Subsubsection \ref{twistedPartyAlgebraS}, we explain how $\P_n(\delta)$ can be regarded as a twisted monoid algebra of $\P_n$, and show how Theorem \ref{TwistedPresentation} recovers its defining presentation. Subsection~\ref{SubSecMaxCell} is devoted to recalling the maximal subgroups of $\P_n$ (Subsubsection \ref{MaxGroupsParty}) and compute those of $T\Sfr_n$ (Proposition \ref{ProMaxSubTSn}), as well as proving the cellularity of $\P_n(\delta)$ (Theorem \ref{Pn(d)cellular}). This cellularity, originally proven in \cite{Kos08}, is obtained here in a different way. More precisely, we utilize the fact that the party algebra is a twisted monoid algebra, as well as \cite[Corollary~6]{Wilcox2007}. The section concludes with two additional subsections, where we introduce two natural algebraic structures that emerge from the previous discussions: the twisted monoid algebra of braids and ties (Subsection~\ref{SubSecTwisBT}) and the party-Brauer-like monoid (Subsection~\ref{SubSecParBrMon}).

Section~\ref{SecParty-Hecke} consists of four subsections. In Subsection~\ref{SubsecParty-Hecke}, we introduce the party braid monoid $\BP_n$, which is a natural \lq party extension\rq\ of the braid group. We define the Party-Hecke algebra in a manner analogous to how the Iwahori--Hecke algebra is defined as a quotient of the group algebra $\CC[B_n]$. Specifically, given $p,q\in\CC^\times$, the Party-Hecke algebra $\Pa_n(p,q)$ is defined as a quotient of $\CC[\BP_n]$ by the ideal generated by the elements\[\sigma_i^2-pq^2-p(p-1)\bar{f}_i,\qquad\sigma_i\bar{f}_i-pq \bar{f}_i,\qquad\bar{f}_i^2-q^2\bar{f}_i,\]where $\sigma_i$ is the elementary braid and $\bar{f}_i$ is the elementary party generator (Definition \ref{Pn(p,q)}). Note that, after a simple rescaling of the generators, it follows that the algebras $\Pa_n(1,q)$ and $\P_n(q)$ coincide. In Subsection~\ref{SubsecParty-HeckeQuot}, we prove that the Party-Hecke algebra can be realized as a quotient of a twisted monoid algebra of the tied braid monoid (Proposition \ref{TBnQuotient}). Subsection~\ref{SubsecBasisParty-Hecke} is dedicated to proving Theorem~\ref{Gnbasis}, which provides a linear basis for $\Pa_n(p,q)$ parametrized by elements of $\P_n$. Moreover, the elements of this linear basis can be factorized into the form $F_IG_s$, where $(I,s)\in P_n\times\Sfr_n$ is a coprime pair (Proposition \ref{partyNF}). The proof of Theorem~\ref{Gnbasis} relies on Proposition~\ref{partyNF}, together with a Jimbo-type tensorial representation for $\Pa_n(p,q)$. Finally, in Subsection \ref{SubsecBasisGenRep}, we prove that the Party-Hecke algebra is generically semisimple.

In Section \ref{SecParty-HeckeQuo}, we prove that the Party-Hecke algebra can be realized as a quotient of the algebra of braids and ties (Proposition~\ref{P-H/bt}). This realization is obtained  via an alternative presentation introduced at the beginning of Subsection~\ref{P'asQuot}. Subsection~\ref{TwoTL} introduces a new presentation for the Party-Hecke algebra, where the braid generators are replaced by idempotents. Motivated by the classical realization of the Temperley--Lieb algebra  as a quotient of the Iwahori--Hecke algebra, we are led to consider two quotients of Temperley--Lieb type which are detailed  in Propositions~\ref{HP/I} and \ref{HP/J}. Furthermore, we define a third quotient of the Party-Hecke algebra of Temperley--Lieb type in the sense of \cite{Juyumaya13,Ry-HaJPAA22}, which appears to be a novel algebraic structure.

Finally, in Section \ref{Future}, we discuss ongoing work in two directions:  the representation theory of $T\Sfr_n$ (Subsection~\ref{RepTSn}) and the  study of a Jones-type invariant for virtual knots (Subsection \ref{Virtual}).

\section{Preliminaries}\label{SecPreliminaries}

Throughout this paper, $d$ and $n$ denote two positive integers. We denote by $\dbrack{a,b}$ the set of integers $x$ such that $a\leq x\leq b$. Furthermore, if a generator's subscript is omitted, it is assumed that this subscript takes any value appropriate for that generator. The remainder of this section provides the preliminaries, which are divided into Subsections \ref{SecSetPartitions} to \ref{SecIwahoriHeckeBT}.

\subsection{Set partitions}\label{SecSetPartitions}
Recall that a \emph{set partition} of a set $A$ is a collection of disjoint nonempty subsets, called \emph{blocks}, such that their union is $A$. The collection of set partitions of $A$ is denoted by $P(A)$. For each nonempty $B\subset A$ we write $f_B$ to denote the set partition whose unique possibly nontrivial block is $B$.

Given $I,J\in P(A)$, we say that $J$ is \emph{coarser} than $I$, or that $I$ is \emph{finer} than $J$, denoted by $I\preceq J$, if each block of $J$ is a union of blocks of $I$. This relation gives $P(A)$ the structure of a lattice with \emph{join operation} $\vee$. Hence, the pair $(P(A),\vee)$ defines an idempotent commutative monoid with identity $(\{a\}\mid a\in A)$, called the \emph{monoid of set partitions} of $A$. For $I,J\in P(A)$, it is usual to write $IJ$ instead of $I\vee J$.

For a subset $B\subset A$ and a set partition $I\in P(A)$, we write $I|_B$ to denote the set partition of $B$ consisting of the nonempty intersections of the blocks of $I$ with $B$. Conversely, if $I\in P(B)$, then it can be regarded as a set partition of $A$ by completing $I$ with the singleton blocks formed by the elements in $A\setminus B$. Accordingly, if $A,B$ are any two sets, and $(I,J)\in P(A)\times P(B)$, we simply write $I\vee J$ to denote the join of $I$ with $J$ regarded as set partitions of $A\cup B$.

As shown in \cite[Theorem 2]{Fi2003}, the monoid of set partitions $P_n:=P(\dbrack{1,n})$ is presented by generators $f_{i,j}:=f_{\{i,j\}}$ with $i<j$, subject to the relations:\begin{gather}f_{i,j}^2=f_{i,j},\qquad f_{i,j}f_{h,k}=f_{h,k}f_{i,j},\qquad f_{i,j}f_{i,k}=f_{i,j}f_{j,k}=f_{i,k}f_{j,k}.\label{PnRels}\end{gather}It is known that $|P_n|$ is the $n$th \emph{Bell number} \cite[\href{https://oeis.org/A000110}{A000110}]{OEIS}. Moreover, the elements of $P_n$ admit the following normal form.
\begin{proposition}[{\cite[Section 3.1]{ArJu2021}}]\label{normalP_n}
Let $I=(I_1,\ldots,I_k)$ be a set partition of $\dbrack{1,n}$. Then $I=f_{I_1}\cdots f_{I_k}$. More precisely, for $B=\{i_1<\cdots<i_k\}\subseteq\dbrack{1,n}$ with $k=|B|>1$, we have:\[f_B=f_{i_1,i_2}\cdots f_{i_{k-1},i_k}.\]
\end{proposition}

The submonoid $C_n$ of $P_n$, generated by the elements $f_i:=f_{i,i+1}$ for all $i\in\dbrack{1,n-1}$, is called the \emph{monoid of compositions} \cite[Subsection 2.1]{ArEs2024}, and is isomorphic to the free idempotent commutative monoid of rank $n-1$, that is, $C_n$ is presented by generators $f_1,\ldots,f_{n-1}$, subject to the relations:
\begin{gather}
f_i^2=f_i,\qquad f_if_j=f_jf_i.\label{Par1}
\end{gather}

\subsection{The partition monoid}
Write $\Cfr_n:=P(\dbrack{1,2n})$, and let $X=\{x_1,\ldots,x_n\}$ be a set of cardinality $n$ which is disjoint with $\dbrack{1,2n}$. A set partition $I\in\Cfr_n$ is usually represented by a so-called \emph{strands diagram}. This diagram is obtained by placing $n$ top points, representing the elements of $\dbrack{1,n}$, and $n$ bottom points, representing the elements of $\dbrack{n+1,2n}$, which are also labeled from $1$ to $n$ for convenience. The points are connected transitively according to the blocks of $I$. See Figure~\ref{str_diag}.\begin{figure}[H]\figfiv\caption{ Strand diagram of a set partition.}\label{str_diag}\end{figure}

The \emph{concatenation} of two set partitions $I,J\in\Cfr_n$ is the set partition\[I*J:=(I_X\vee J^X)|_{\dbrack{1,2n}},\]where $I_X$ is the set partition obtained from $I$ by replacing each $i\in\dbrack{1,n}$ with $x_i$, and $J^X$ is the set partition obtained from $J$ by replacing each $n+i\in\dbrack{n+1,2n}$ with $x_i$. In terms of diagrams, this concatenation is obtained by identifying the bottom points of $I$ with the top points of $J$. This product gives to $\Cfr_n$ the structure of a monoid with identity $(\{i,n+i\}\mid i\in\dbrack{1,n})$, called the \emph{partition monoid} \cite{Martin1994,Jones94}.

The monoid $P_n$ can be regarded as a submonoid of $\Cfr_n$ by identifying each $f_{i,j}$
with the set partition whose blocks are $\{i,j,n+i,n+j\}$ and $\{k,n+k\}$ for all $k\neq i,j$. See Figure~\ref{diagP_3}.
\begin{figure}[H]
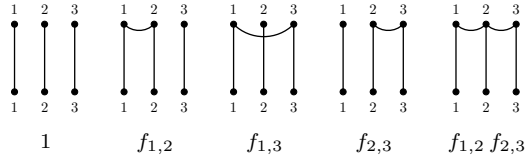

\figone\caption{The 5 elements of $P_3$.}\label{diagP_3}
\end{figure}

\subsection{The symmetric group and the braid group}\label{SubSecSymmetricBraidGroups}

\subsubsection{}As shown in \cite{Moo896}, the {\it symmetric group} $\Sfr_n$, of permutations of $\dbrack{1,n}$, is presented by generators $s_1,\ldots,s_{n-1}$, subject to the following relations:
\begin{gather}
s_i^2=1,\qquad s_is_js_i=s_js_js_i\quad\text{if }|i-j|=1,\qquad s_is_j=s_js_i\quad\text{if }|i-j|>1.\label{TSn1}
\end{gather}
Here, each $s_i$ corresponds to the elementary transposition $(i\,\,i+1)$. It is well known that $\Sfr_n$ can be regarded as a submonoid of $\Cfr_n$ by identifying each $s_i$ with the set partition whose blocks are $\{i,n+i+1\}$, $\{i+1,n+i\}$ and $\{k,n+k\}$ for all $k\neq i,i+1$. Moreover, $\Sfr_n$ is the group of units of the partition monoid $\Cfr_n$. See Figure~\ref{diagS_3}.
\begin{figure}[H]
\figzer\caption{The 6 elements of $\Sfr_3$.}\label{diagS_3}
\end{figure}

By applying \emph{Tietze transformations} \cite[Chapter 3]{Rus1995}, we observe that the symmetric group $\Sfr_n$ can also be presented by generators $s_{i,j}$ with $i,j\in\dbrack{1,n}$ and $i<j$, subject to the following relations:
\begin{gather}
s_{i,j}^2=1,\qquad s_{i,j}s_{j,k}=s_{j,k}s_{i,k}=s_{i,k}s_{i,j}.\label{DualSn1}
\end{gather}
Here, each generator $s_{i,j}$ represents the transposition swapping $i$ with $j$, defined by\[s_{i,j}=s_{j-1}\cdots s_{i+1}s_is_{i+1}\cdots s_{j-1}\]for all $i<j$. One can deduce from \eqref{DualSn1} the following superfluous relations:\begin{gather}
s_{i,j}s_{i,k}=s_{j,k}s_{i,j}=s_{i,k}s_{j,k},\qquad s_{i,j}s_{r,s}=s_{r,s}s_{i,j}\quad\text{if }\{i,j\}\cap\{r,s\}=\emptyset.\label{DualSn2}
\end{gather}

\subsubsection{}As usual we denote by $B_n$ the \emph{braid group} on $n$ strands \cite{Artin1925}, that is, the group presented by generators $\sigma_1,\ldots,\sigma_{n-1}$, subject to the \emph{braid relations}:
\begin{gather}
\sigma_i\sigma_j\sigma_i=\sigma_j\sigma_i\sigma_j\quad\text{if }|i-j|=1,\qquad \sigma_i\sigma_j=\sigma_j\sigma_i\quad\text{if }|i-j|>1.\label{braidRels}
\end{gather}
 Observe that $\Sfr_n$ is the quotient of $B_n$ by the relations $\sigma_i^2=1$. As shown in \cite{BiKoLee1998}, the braid group can also be presented by generators
\begin{equation}
\!\sigma_{i,j}:=\sigma_i\cdots\sigma_{j-2}\sigma_{j-1}\sigma_{j-2}^{-1}\cdots\sigma_i^{-1}\quad\text{for all}\quad i,j\in\dbrack{1,n}\quad\text{with}\quad i<j,\label{genSigma}
\end{equation}
subject to the \emph{dual braid relations}:
\begin{gather}
\sigma_{i,j}\sigma_{j,k}=\sigma_{j,k}\sigma_{i,k}=\sigma_{i,k}\sigma_{i,j}\quad\text{if }i<j<k,\label{dualBraidRels1}\\
\sigma_{i,j}\sigma_{r,s}=\sigma_{r,s}\sigma_{i,j}\quad\text{if }i<j<r<s\text{ or }i<r<s<j.\label{dualBraidRels2}
\end{gather}
The \emph{braid monoid} on $n$ strands \cite{Gar1969} is the submonoid $B_n^+$ of $B_n$ generated by $\sigma_1,\ldots,\sigma_{n-1}$, which is also presented by the braid relations in \eqref{braidRels}.

\subsection{Ramified monoids}\label{SubSecRamifiedMon}

Given a submonoid $M$ of $\Cfr_n$, the \emph{ramified monoid} $\R M$ of $M$ \cite{AiArJu23,AiArJu24} is defined as the submonoid of $M\times \Cfr_n$ consisting of the pairs $(I,J)$ such that $I\preceq J$. Each element $(I,J)\in\R M$ can be represented by a strand diagram, which is obtained by adding to the diagram of $I$ wavy lines between the blocks of $I$ that originate the blocks of $J$.

Notice that the monoid $P_n$ can also be regarded as a submonoid of $\RC_n$. Indeed, $P_n$ is isomorphic to $\R\{1\}$ by means of the map $f_{i,j}\mapsto e_{i,j}$ for all $i<j$, where each $e_{i,j}$ is the ramified set partition $(1,f_{i,j})$ \cite[Proposition 2]{AiArJu24}. See Figure~\ref{eijStrand}. Consequently, $P_n$ is a submonoid of $\R M$, for any submonoid $M\subseteq\Cfr_n$.\begin{figure}[H]
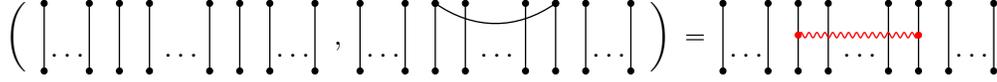
\figele\caption{Generator $(1,f_{i,j})=e_{i,j}.$}\label{eijStrand}\end{figure}

\subsubsection{}\label{RamifiedSymmetricMonoid}The ramified monoid $\R\Sfr_n$ is known to be isomorphic to the so-called \emph{tied symmetric monoid} $T\Sfr_n$ \cite[Theorem 1]{AiArJu23}. As shown in \cite[Subsection 5.1]{AiArJu23}, $T\Sfr_n$ admits a presentation by generators $s_1,\ldots,s_{n-1}$ satisfying \eqref{TSn1}, and generators $e_1,\ldots,e_{n-1}$, subject to the following relations:
\begin{gather}
e_i^2=e_i,\qquad e_ie_j=e_je_i,\label{TSn2}\\
 s_is_je_i=e_js_is_j\quad\!\text{and}\!\quad s_ie_je_i=e_js_ie_j=e_ie_js_i\quad\text{if }|i-j|=1,\qquad\!\!s_ie_j=e_js_i\quad\text{if }|i-j|\neq1.\label{TSn3}
\end{gather}
Note that the relations in \eqref{TSn1} form a closed set of relations between the $s_i$'s, which in fact define the symmetric group $\Sfr_n$ as a subgroup of $T\Sfr_n$. Similarly, the relations in \eqref{TSn2} form a closed set of relations that define the monoid of compositions $C_n$ as a submonoid of $T\Sfr_n$. It is known that $T\Sfr_n=P_n\rtimes\Sfr_n$ \cite[Remark 5]{AiArJu24}, where $\Sfr_n$ acts on $P_n$ by $s(I):=(s(I_1),\ldots,s(I_k))$. Here, $P_n$ is embedded inside $T\Sfr_n$ by means of the map $f_{i,j}\mapsto e_{i,j}$ for all $i,j\in\dbrack{1,n}$ with $i<j$, where:\[e_{i,j}=s_i\cdots s_{j-2}e_{j-1}s_{j-2}\cdots s_i=s_{j-1}\cdots s_{i+1}e_is_{i+1}\cdots s_{j-1}.\]Indeed, applying \cite[Corollary 2]{Lavers1998}, the monoid $T\Sfr_n$ can also be presented by generators $s_1,\ldots,s_{n-1}$, satisfying \eqref{TSn1}, and generators $e_{i,j}$ with $i<j$, subject to the following relations:
\begin{gather}
e_{i,j}^2=e_{i,j},\qquad e_{i,j}e_{r,s}=e_{r,s}e_{i,j},\qquad e_{i,j}e_{i,k}=e_{i,j}e_{j,k}=e_{i,k}e_{j,k},
\label{TSn5}\\
s_ie_{j,k}=e_{s_i(j),s_i(k)}s_i,\quad\text{where}\quad e_{j,i}=e_{i,j}.\label{TSn6}
\end{gather}

\subsubsection{}\label{SubSubSecTiedBraidMon}The \emph{tied braid monoid} \cite[Section 3]{AiJu16} is the monoid $TB_n$ presented by generators $\sigma_1,\ldots,\sigma_{n-1}$ satisfying \eqref{braidRels}, generators $\sigma_1^{-1},\ldots,\sigma_{n-1}^{-1}$, and generators $e_1,\ldots,e_{n-1}$ satisfying \eqref{TSn2}, subject to the following relations:
\begin{gather}
\sigma_i\sigma_i^{-1}=\sigma_i^{-1}\sigma_i=1,\label{TBM1}\\
\sigma_ie_j=e_j\sigma_i\quad\text{if }|i-j|=1,\qquad\sigma_i\sigma_j^{\pm1}e_i=e_j\sigma_i\sigma_j^{\pm1}\quad\text{if }|i-j|=1,\label{TBM2}\\
\sigma_ie_je_i=e_j\sigma_ie_j=e_ie_j\sigma_i\quad\text{if }|i-j|=1.\label{TBM3}
\end{gather}
Observe that $T\Sfr_n$ is the quotient of $TB_n$ by the relations $\sigma_i^2=1$. It is known that $TB_n=P_n\rtimes B_n$ \cite[Theorem 3]{AiJu21}. This implies that each element of $TB_n$ can be uniquely written as a product $e\beta$, where $e\in P_n$ and $\beta\in B_n$. In this context, $P_n$ is embedded inside $TB_n$ by means of the map $f_{i,j}\mapsto e_{i,j}$ for all $i,j\in\dbrack{1,n}$, where
\begin{gather}
e_{i,j}=\sigma_i\cdots\sigma_{j-2}e_{j-1}\sigma_{j-2}^{-1}\cdots\sigma_i^{-1}=\sigma_{j-1}^{-1}\cdots\sigma_{i+1}^{-1}e_i\sigma_{i+1}\cdots\sigma_{j-1}.\label{eijTBn}
\end{gather}
See \cite[Subsection 3.2]{AiJu21} for more details. Indeed, by applying \cite[Corollary 2]{Lavers1998}, the monoid $TB_n$ can also be presented by generators $\sigma_1,\ldots,\sigma_{n-1}$ satisfying \eqref{braidRels}, generators $\sigma_1^{-1},\ldots,\sigma_{n-1}^{-1}$ satisfying \eqref{TBM1}, and generators $e_{i,j}$ with $i<j$ satisfying \eqref{TSn5}, subject to the following relations:
\begin{gather}
\label{actionRelTBn}\sigma_ie_{j,k}=e_{s_i(j),s_i(k)}\sigma_i,\quad\text{where}\quad e_{j,i}=e_{i,j}.
\end{gather}

\subsubsection{}\label{BrauerSSS}
The \emph{Brauer monoid} $\Brfr_n$ is the submonoid of $\Cfr_n$ consisting of the set partitions whose blocks have exactly two elements \cite{Brauer1937}. According to \cite[Theorem 3.1]{KuMaCEJM2006}, the monoid $\Brfr_n$ admits a presentation by generators $s_1,\ldots,s_{n-1}$, $t_1,\ldots,t_{n-1}$ satisfying \eqref{TSn1}, together with the relations:
\begin{gather}
t_i^2=t_i,\qquad t_it_jt_i=t_i\quad \text{if }|i-j|=1,\qquad t_it_j=t_jt_i\quad \text{if }|i-j|>1,\label{Br1}\\
s_it_i=t_i,\qquad\!\!s_jt_i=t_is_j\quad\text{if }|i-j|\neq1,\qquad\!\!s_it_jt_i=s_jt_i\quad\text{and}\quad t_it_js_i=t_is_j\quad \text{if }|i-j|=1.\label{Br2}
\end{gather}
In this context, each $t_i$ corresponds to the set partition whose blocks are $\{i,i+1\}$, $\{n+i,n+i+1\}$ and $\{k,n+k\}$ for all $k\neq i,i+1$. See Figure~\ref{diagTangleGens}.\begin{figure}[H]
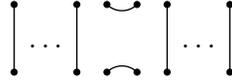
\figsix\caption{Generator $t_i$.}\label{diagTangleGens}\end{figure} The submonoid of $\Brfr_n$ generated by $t_1,\ldots,t_{n-1}$ subject to the relations in \eqref{Br1} is the well-known \emph{Jones monoid} $\Jfr_n$ \cite{TL1971,Jones83}.

\subsubsection{}\label{SubSubSec244}
It was shown in \cite[Theorem 42]{AiArJu23} that the ramified monoid $\R\Brfr_n$ is presented by generators $s_1,\ldots,s_{n-1}$, $e_1,\ldots,e_{n-1}$, $t_1,\ldots, t_{n-1}$ and $d_1,\ldots,d_{n-1}$ satisfying \eqref{TSn1}, \eqref{Br1}--\eqref{Br2}, \eqref{TSn2}--\eqref{TSn3}, and the following relations:
\begin{gather}
d_i^2=d_i,\qquad d_id_j=d_jd_i\quad\text{if }|i-j|>1,\label{RelsRBrn1}\\
e_id_i=d_i,\qquad e_id_j=d_je_i,\qquad e_jd_ie_j=d_id_jd_i\quad\text{if }|i-j|=1,\label{RelsRBrn2}\\
s_id_i=d_i,\qquad s_id_j=d_js_i\quad\text{if }|i-j|\neq1,\qquad s_is_jd_i=d_js_is_j\quad\text{if }|i-j|=1,\label{RelsRBrn3}\\
t_ie_i=t_i,\qquad t_ie_j=e_jt_i\quad\text{if }|i-j|\neq1,\qquad t_id_i=t_i,\qquad t_id_j=d_jt_i\quad\text{if }|i-j|\neq1,\label{RelsRBrn4}\\
e_jt_ie_j=e_jd_i\quad\text{if }|i-j|=1.\label{RelsRBrn5}
\end{gather} Here, each $d_i$ corresponds to the ramified set partition $(t_i,f_i)$. See Figure~\ref{diagTiedTangleGens}.\begin{figure}[H]
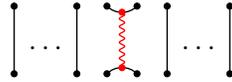
\figttn\caption{Generator $d_i$.}\label{diagTiedTangleGens}\end{figure}
\begin{remark}
A presentation for the ramified monoid $\R\Jfr_n$ is unknown, and finding one remains an open problem.
\end{remark}
\subsection{Twisted monoid algebras and the partition algebra}\label{SubSecTwiMonAlg}

\subsubsection{}A \emph{twisting} from a monoid $M$ into $\CC$ is a map $\tau:M\times M\to\CC$ satisfying $\tau(x,y)\tau(xy,z)=\tau(x,yz)\tau(y,z)$ for all $x,y,z$. The \emph{twisted monoid algebra} of $M$, denoted by $\CC^\tau[M]$, is the $\CC$-vector space spanned by $M$ together with the product $x\cdot y:=\tau(x,y)xy$ for all $x,y\in M$ \cite[Section 3]{Wilcox2007}. As mentioned in \cite{Wilcox2007}, if $N\subset M$ is submonoid, then $\tau$ restricted to $N\times N$ is a twisting from $N$ into $\CC$, and the twisted monoid algebra of $N$ with respect to this twisting is precisely the subalgebra of $\CC^\tau[M]$ spanned by $N$, which we simply denote by $\CC^\tau[N]$.

\begin{theorem}[{\cite[Theorem 44]{East2011}}]\label{TwistedPresentation}
Assume that $M=\langle X\mid R\rangle$. Then, $\CC^\tau[M]=\langle X\mid\hat{R}\rangle$, where $\hat{R}=\{(\tau([v])u,\tau([u])v)\mid(u,v)\in R\}$.
\end{theorem}

\subsubsection{}Let $\delta\in\CC^\times$. The \emph{partition algebra} $\Cfr_n(\delta)$ \cite{Jones94,Martin1994} is the twisted monoid algebra of the monoid $\Cfr_n$ with respect to the twisting $\tau_\alpha(I,J):=\delta^{\,\alpha(I,J)}$, where $\alpha(I,J)$ is the number of blocks in $I_X\vee J^X$ that are contained in $X$ \cite[Section 7]{Wilcox2007}.

\subsubsection{}The \emph{Temperley--Lieb algebra} $TL_n(\delta)$ \cite{TL1971,Jones83} is defined as the twisted monoid algebra associated to the Jones monoid $\Jfr_n$ with respect to $\tau_\alpha$. By Theorem~\ref{TwistedPresentation}, $TL_n(\delta)$ is the subalgebra of $\Cfr_n(\delta)$ presented by generators $t_1,\ldots,t_{n-1}$, known as \emph{tangles}, subject to the following relations:\begin{gather}
t_i^2=\delta t_i,\qquad t_it_jt_i=t_i\quad\text{if }|i-j|=1,\qquad t_it_j=t_jt_i\quad\text{if }|i-j|>1.
\end{gather} Recall that each $t_i$ corresponds to the set partition depicted in Figure~\ref{diagTangleGens}. 

\subsection{The Iwahori--Hecke algebra and the algebra of braids and ties}\label{SecIwahoriHeckeBT}

Let $u,v$ be two indeterminates in $\CC^\times.$

\subsubsection{}The \emph{Iwahori--Hecke algebra} $\H_n(u)$ \cite{Iwahori64} is defined as the quotient of the group algebra $\CC[B_n]$ by the two-sided ideal generated by the elements:\begin{equation}
\sigma_i^2-u-(u-1)\sigma_i.
\end{equation}
If we let $g_i$ be the image of the generator $\sigma_i$ under this quotient, the algebra $\H_n(u)$ admits a presentation by generators $g_1,\ldots,g_{n-1}$, subject to the following relations:\begin{gather}
g_ig_jg_i=g_jg_ig_j\quad\text{for }|i-j|=1,\quad g_ig_j=g_jg_i\quad\text{for }|i-j|>1,\label{heckeRel1}\\ g_i^2=u+(u-1)g_i\label{heckeRel2}
\end{gather}
It is known that $TL_n(\delta)$ with $\delta^2=(u+1)^2u^{-1}$ can be realized as a quotient of $\H_n(u)$. Indeed, by defining $h_i=(u+1)^{-1}(g_i+1)$, the algebra $\H_n(u)$ can be presented by generators $h_1,\ldots,h_{n-1}$, subject to the following relations:\begin{gather}
h_i^2=h_i,\quad h_ih_j=h_jh_i\quad\text{for }|i-j|>1,\label{heckeRel3}\\
h_ih_jh_i-\delta^{-1}h_i=h_jh_ih_j-\delta^{-1}h_j\quad\text{for }|i-j|=1.\label{heckeRel4}
\end{gather}Taking the quotient by the two-sided ideal generated by the elements $h_ih_jh_i-\delta^{-1}h_i$, and defining $t_i=\delta h_i$, we recover the Temperley--Lieb algebra $TL_n(\delta)$.

\subsubsection{}The \emph{algebra of braids and ties} $\E_n(u,v)$ \cite[Section 1]{AiJu20}, or simply the \emph{bt-algebra}, is defined as the quotient of the monoid algebra $\CC[TB_n]$ by the two-sided ideal generated by the elements:\begin{equation}\sigma_i^2-1-(u-1)e_i-(v-1)e_i\sigma_i.\end{equation}Consequently, if we denote by $g_i$ the image of the generator $\sigma_i$ under this quotient, the algebra $\E_n(u,v)$ admits a presentation by generators $g_1,\ldots,g_{n-1}$ satisfying \eqref{heckeRel1}, and generators $e_1,\ldots,e_{n-1}$ satisfying \eqref{TSn2}, subject to the following relations: 
\begin{gather}
g_ie_i=e_ig_i,\quad g_ie_j=e_jg_i\quad\text{for }|i-j|>1,\label{bt3}\\
g_ig_je_i=e_jg_ig_j\quad\text{and}\quad g_ie_je_i=e_jg_ie_j=e_ie_jg_i,\quad\text{for }|i-j|=1,\label{bt4}\\
g_i^2=1+(u-1)e_i+(v-1)e_ig_i.\label{bt5}
\end{gather}
We denote by $\E_n(u)$ the algebra $\E_n(u,u)$, and by $\E_n'(v)$ the algebra $\E_n(u,v)$ with $v^2=u$. Observe that the Iwahori--Hecke algebra $\H_n(u)$ is the quotient of the algebra $\E_n(u)$ by the two-sided ideal generated by the elements $e_i-1$ for all $i\in\dbrack{1,n-1}$.

\section{Party monoids and algebras}\label{SecParMonAlg}

Here, we recall the definitions of the party monoid and the party algebra. We then provide a new proof of the cellularity of the party algebra, by applying a result by Wilcox \cite{Wilcox2007}. This approach relies on viewing the party algebra as a twisted monoid algebra and analyzing its maximal subgroups. Finally, we introduce two new structures: a twisted monoid algebra of the tied braid monoid, and the Party-Brauer-like monoid.

\subsection{The party monoid and the party algebra}

\subsubsection{}\label{subPartyMon}The {\it party monoid} \footnote{Also known as the \emph{monoid of uniform block permutations} \cite[Section 2.2]{AgOr2008}.} $\P_n$ {\cite[Theorem 1.1]{Kos00}} is the submonoid of $\Cfr_n$ presented by generators $s_1,\ldots,s_{n-1}$ satisfying \eqref{TSn1}, and generators $f_1,\ldots,f_{n-1}$ satisfying \eqref{Par1}, both subject to the relations:
\begin{align}
s_if_i=f_is_i=f_i,&\quad s_if_j=f_js_i\quad\text{if $|i-j|>1$},\label{Par2}\\
s_is_jf_i & =f_js_is_j\quad\text{if $|i-j|=1$.}\label{Par3}
\end{align}

As with $T\Sfr_n$, the party monoid $\P_n$ contains a copy of the symmetric group as the submonoid generated by $s_1,\ldots,s_{n-1}$. Moreover, it is known that $\P_n=P_n\Sfr_n=\Sfr_nP_n$ \cite[Subsection~4.2]{AgOr2008} \cite[Subsection~2.5]{OrSaSchZaAl2022}, where $P_n$ coincides with the submonoid of idempotents of $\P_n$. In this context, $P_n$ is regarded as the submonoid generated by $f_{i,j}:=s_{i+1,j}f_is_{i+1,j}$ with $i<j$. More precisely, for every $g\in\P_n$ there exist unique set partitions $f,f'\in P_n$ such that $g=fs=sf'$ for some $s\in\Sfr_n$ \cite[Subsubsection~2.5.3]{OrSaSchZaAl2022}. The cardinality of $\P_n$ corresponds to the sequence \cite[\href{https://oeis.org/A023998}{A023998}]{OEIS}. For instance, it equals $3$, $16$, $131$ and $1496$ for $n=2,3,4$ and $5$, respectively. For further details, we refer the reader to \cite[Subsection 2.3]{OrSaSchZaAl2022}. See Figure~\ref{fig1} and Figure~\ref{diagParty_3}.

\begin{figure}[H]
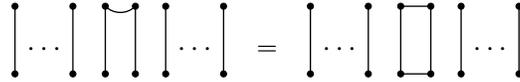
\figfou\caption{Two equivalent diagrammatic representations for $f_i$.}\label{fig1}\end{figure}

\begin{figure}[H]
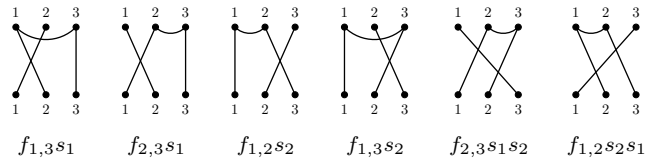
\figtwo\caption{These elements together with those in Figure~\ref{diagS_3} and Figure~\ref{diagP_3} form $\P_3$.}\label{diagParty_3}\end{figure}

\begin{proposition}\label{TSn/R}
$\P_n$ is the quotient $T\Sfr_n/R$, where $R$ is the congruence generated by the relations $s_ie_i=e_i$ for all $i$.
\end{proposition}
\begin{proof}
Observe that $s_if_if_j=f_if_js_i=f_js_if_j$ holds in $\P_n$ for all $i,j$ with $|i-j|=1$. Indeed, by applying \eqref{Par1} and \eqref{Par2}, for each $|i-j|=1$, we get\[s_if_jf_i=s_if_if_j=f_if_j=f_jf_i=f_jf_is_i=f_if_js_i.\]Now, due to \eqref{Par3} and the equation above, we obtain\[f_js_if_j=f_js_jf_is_js_i=f_jf_is_js_i=f_if_js_js_i=f_if_js_i=f_if_j.\]This shows that $\P_n$ is $T\Sfr_n/R$.
\end{proof}
By applying simple Tietze transformations \cite[Chapter 3]{Rus1995}, we can also present the party monoid $\P_n$ by generators $s_{i,j}$ with $i,j\in\dbrack{1,n-1}$ and $i<j$, satisfying \eqref{DualSn1}, and generators $f_{i,j}$ with $i<j$, satisfying \eqref{PnRels}, both subject to the following relations:\begin{gather}
s_{i,j}f_{h,k}=f_{s_{i,j}(h),s_{i,j}(k)}s_{i,j},\quad\text{where}\quad f_{j,i}=f_{i,j},\label{PartyDual1}\\
f_{i,j}s_{i,j}=f_{i,j}.\label{PartyDual2}
\end{gather}

Since $T\Sfr_n=P_n\rtimes\Sfr_n$, every element of $T\Sfr_n$ can be written uniquely as a product $es$, where $e\in P_n$ and $s\in\Sfr_n$. However, since $f_i=f_is_i$ for all $i\in\dbrack{1,n-1}$, this is not true for $\P_n$, even though $\P_n=P_n\Sfr_n$. Nevertheless, we aim to obtain a normal form similar to the one mentioned above for $T\Sfr_n$.

Recall that for $s,s'\in\Sfr_n$, we say that $s$ is a \emph{left divisor} of $s'$ whenever $s'=st$ with $\ell(s')=\ell(s)+\ell(t)$ for some $t\in\Sfr_n$. Similarly, for $f,f'\in P_n$, we will say that $f$ is a \emph{divisor} of $f'$ if $f'=f'f$. A pair $(f,s')\in P_n\times\Sfr_n$ is called \emph{coprime} if there are no left divisors $s$ of $s'$ such that $fs=f$. 

\begin{proposition}[Normal form]\label{partyNF}
For every $g\in\P_n$ there is a unique coprime pair $(f,s')\in P_n\times\Sfr_n$ such that $fs'=g$.
\end{proposition}
\begin{proof}
Due to Proposition~\ref{TSn/R} and \eqref{PartyDual2}, there is a unique $f\in P_n$, and there is a permutation $s'\in\Sfr_n$, such that $fs'=g$ with $(f,s')$ coprime. Now, suppose there is $s\in\Sfr_n$ such that $g=fs$ and $(f,s)$ is coprime, then $fs=fs'$. Hence, we get $f=fs's^{-1}$, thus, either $s's^{-1}=1$ or $s's^{-1}=s_{i_1,j_1}\cdots s_{i_k,j_k}$, where $f_{i_1,j_1},\ldots,f_{i_k,j_k}$ are all divisors of $f$. This implies that $s'=s_{i_1,j_1}\cdots s_{i_k,j_k}s$, which is a contradiction because $(f,s)$ is coprime. Therefore $s=s'$.
\end{proof}

\begin{remark}
Recall that the length of a permutation $s$ coincides with the number of its inversions. Thus, if $(i,j)$ is an inversion of $s$, then $\ell(s_{i,j}s)<\ell(s)$. See \cite[Subsection 5.8]{Humph1997} for details. Therefore, $(i,j)$ is an inversion of $s$ if and only if $s_{i,j}$ is a proper left divisor of $s$. If $g=fs\in\P_n$ for some $f\in P_n$, then we can get the normal form of $g$ by applying \eqref{PartyDual2} repeatedly. See Figure~\ref{diag_normal}.
\end{remark}
\begin{figure}[H]
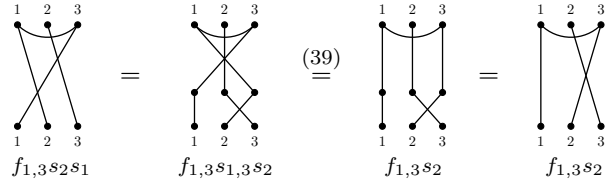
\figthr\caption{ Normal form of $f_{1,3}s_2s_1$.}\label{diag_normal}\end{figure}

\subsubsection{}\label{SubSubSec312}Given $\delta\in\CC^\times$, the \emph{party algebra} $\P_n(\delta)$ \cite{Kos00} is the $\CC$-algebra presented by generators $S_1,\ldots,S_{n-1}$ and $F_1,\ldots,F_{n-1}$ subject to the following relations:
\begin{gather}
S_i^2=1,\qquad S_iS_jS_i=S_jS_jS_i\quad\text{if }|i-j|=1,\qquad S_iS_j=S_jS_i\quad\text{if }|i-j|>1,\label{partyAlgRel1}\\
F_i^2=\delta F_i,\qquad F_iF_j=F_jF_i,\label{partyAlgRel2}\\
S_iF_i=F_iS_i=F_i,\qquad S_iF_j=F_jS_i\quad\text{if }|i-j|>1,\label{partyAlgRel3}\\
S_iS_jF_i=F_jS_iS_j\quad\text{if }|i-j|=1.\label{partyAlgRel4}
\end{gather}
As characterized in \cite[Section 1]{Kos00}, $\P_n(\delta)$ can be regarded as an algebra spanned by the party monoid, where the product is described in terms of set partitions. 
To make this explicit, we introduce some terminology. Following \cite[Section~3]{ArEsFlo2025}, the \emph{arc decomposition} of a finite subset $B=\{q_1<\cdots<q_k\}\subseteq[n]$ is the collection $\hat{B}$ consisting of the arcs $\{q_i,q_{i+1}\}$ for all $i\in[k-1]$. For a set partition $I=(I_1,\ldots,I_k)\in P_n$, we define $\hat{I}=\hat{I}_1\cup\cdots\cup\hat{I}_k$ and call the elements of $\hat{I}$ the \emph{standard arcs} of $I$. Note that $|\hat{B}|=|B|-1$, and therefore\[|\hat{I}|=\sum_{B\in I}\hat{B}=\sum_{B\in I}(|B|-1)=\sum_{B\in I}|B|-|I|=n-|I|.\]Moreover, a pair $\{i,j\}$ with $i<j$ is a standard arc of a set partition $I$ if and only if $e_{i,j}$ is a generator occurring in the normal form of $I$ given in Proposition~\ref{normalP_n}.

\subsubsection{}\label{twistedPartyAlgebraS}
Now, let $g,g'\in\P_n$ be two basis elements of $\P_n(\delta)$. As explained in Subsubsection~\ref{subPartyMon}, there are permutations $s,s'\in\Sfr_n$ and unique set partitions $f,f'\in P_n$ such that $g=sf$ and $g'=f's'$. Then, as described in \cite[Section 1]{Kos00}, their product $gg'$ in $\P_n(\delta)$ is $\delta^{\beta(f,f')}gg'$, where $\beta(f,f')$ denotes the number of common standard arcs of $f$ and $f'$, that is, $\beta(f,f')=|\hat{f}\cap\hat{f}'|$. Equivalently, $\beta(f,f')$ is the number of common generators appearing in the normal forms of $f$ and $f'$. Observe that $\beta(f,f')=\beta(f',f)$ for all $f,f'\in P_n$.

Since the multiplication in $\P_n(\delta)$ is associative, the map $\tau_\beta:\P_n\times\P_n\to\CC$, given by $\tau_\beta(g,g)=\delta^{\beta(f,f')}$, defines a twisting from $\P_n$ into $\CC$. Consequently, the party algebra is precisely the twisted monoid algebra of the party monoid with respect to $\tau_\beta$. By applying Theorem~\ref{TwistedPresentation}, and using the fact that $\beta(f_i,f_i)=1$ for all $i\in[n-1]$, we recover the presentation of $\P_n(\delta)$ given above in terms of the generators of the party monoid.

\begin{remark}
Observe that for any permutations $s,s'\in\Sfr_n\subset\P_n(\delta)$, their set partition components reduce to the identity partition, which consists solely of singleton blocks. Thus, they possess no standard arcs, yielding $\beta(1,1)=0$. Consequently, the twisting restricts to $\tau_\beta(s,s')=1$, which implies that the subalgebra of $\P_n(\delta)$ spanned by $\Sfr_n$ is precisely the group algebra $\CC[\Sfr_n]$.
\end{remark}
\subsection{Maximal subgroups and cellular structure}\label{SubSecMaxCell}
 The main objective of this subsection is to prove that $\P_n(\delta)$ is cellular (Theorem \ref{Pn(d)cellular}). To this end, we first recall the structure of the maximal subgroups of the party monoid $\P_n$ \cite[Subsection 3.1]{OrSaSchZaAl2022} and show how these results extend to determine the maximal subgroups of $T\Sfr_n$. Finally, by leveraging this characterization and applying the framework developed by Wilcox for twisted monoid algebras \cite[Section 5]{Wilcox2007}, we establish the cellularity of the party algebra $\P_n(\delta)$.

Recall that an \emph{inverse monoid} is a monoid $M$ such that for every $g\in M$, there is a unique element $g^*\in M$, called the \emph{inverse} of $g$, satisfying\[gg^*g=g\quad\text{and}\quad g^*gg^*=g^*.\]

If $e$ is an idempotent in $M$, then the subsemigroup $eMe$ forms a monoid with identity $e$. The group of units $G_e:=(eMe)^\times$ is called the \emph{maximal subgroup} of $M$ at $e$. As shown in \cite[Corollary 3.6]{Steinberg2016}, when $M$ is an inverse monoid, one has\[G_e=\{g\in M\mid gg^*=e=g^*g\}.\]

\subsubsection{Maximal subgroups of the party monoid}\label{MaxGroupsParty}

It was shown in \cite[Subsection 4.2]{AgOr2008} and \cite[Proposition 2.8(1)]{OrSaSchZaAl2022} that $\P_n$ is an inverse monoid. Specifically, if $g=fs=sf'\in\P_n$, then its inverse is given by $g^*=s^{-1}f=f's^{-1}$. The mapping $g\mapsto g^*$ defines an anti-involution on $\P_n$; indeed, for any $g=fs$ and $h=f's'$ in $\P_n$, we have $(gh)^*=(fs(f')ss')^*=s'^{-1}s^{-1}fs(f')=s'^{-1}f's^{-1}f=h^*g^*$. This coincides with the anti-involution in \cite[Section 7]{Wilcox2007}, so we have $g^*=(\{k^*\mid k\in B\}\mid B\in g)$, where $k^*:=k+n$ if $k\leq n$ and $k^*:=k-n$ if $k>n$, for all $k\in\dbrack{1,2n}$. See Figure~\ref{inverseElemParty}.
\begin{figure}[H]
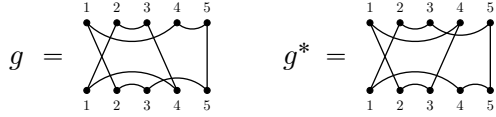
\figsev\caption{An element in $\P_5$ and its inverse.}\label{inverseElemParty}\end{figure}

Given disjoint subsets $A,B\subset\dbrack{1,n}$ with $|A|=|B|$, we denote by $s_{A,B}$ the unique permutation in $\Sfr_n$ that is order-preserving on both $A$ and $B$, and satisfies\begin{equation}s_{A,B}(A)=B,\quad\text{and}\quad s_{A,B}(x)=x\quad\text{for all}\quad x\notin A\cup B.\label{defsAB}\end{equation}Note that, by definition, $s_{A,B}(B)=A$. See Figure~\ref{sABelem}.\begin{figure}[H]
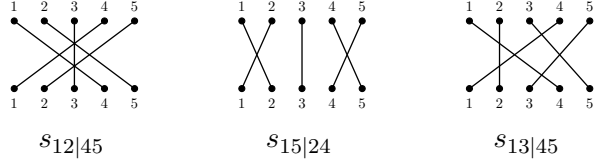
\fignin\caption{Three permutations $s_{A,B}\in\Sfr_5$ with $|A|=|B|=2$.}\label{sABelem}
\end{figure}

Recall that a partition of $n$ is any finite sequence of positive integers, sorted in non-increasing order, such that their sum is $n$. The set of partitions of $n$ is denoted by $\Par_n$.

For a set partition $I\in P_n$, define $\|I\|$ as the partition of block sizes sorted in nonincreasing order, and $\dbrace{I}$ as the set of distinct block sizes. For each $m\in\dbrace{I}$, set $I_{[m]}=\{B\in I\mid|B|=m\}$, and define $\Sfr_{[m]}$ as the subgroup of $\Sfr_n$ generated by the permutations $s_{A,B}$ with $A,B\in I_{[m]}$ satisfying $\min(A)<\min(B)$.

\begin{remark}\label{orbitPartition}
Two set partitions $I,J\in P_n$ belong to the same orbit under the action of $\Sfr_n$ (Subsubsection~\ref{RamifiedSymmetricMonoid}) if and only if $\|I\|=\|J\|$.
\end{remark}

As established in \cite[Proposition 3.2(1) and Corollary 3.3]{OrSaSchZaAl2022}, the maximal subgroup of $\P_n$ at an idempotent $f\in P_n$ is the subgroup $f\Sfr_{[f]}\simeq\Sfr_{[f]}$, where\begin{equation}\label{maxSubParty}\Sfr_{[f]}:=\Sfr_{[m_1]}\times\cdots\times\Sfr_{[m_q]}\simeq\Sfr_{|I_{[m_1]}|}\times\cdots\times\Sfr_{|I_{[m_q]}|},\quad\text{with}\quad\dbrace{f}=\{m_1<\cdots<m_q\}.\end{equation}Moreover, by \cite[Corollary 3.4]{OrSaSchZaAl2022}, the subgroups $\Sfr_{[f]}$ and $\Sfr_{[f']}$ are isomorphic whenever $\|f\|=\|f'\|$, that is, if there is a permutation $s\in\Sfr_n$ such that $\Sfr_{[f]}s=s\Sfr_{[f']}$.

For instance, if $f=(\{1,5\},\{2,6,11\},\{3,4\},\{7,9,12\},\{8,10\})\in P_{12}$, then $\dbrace{I}=\{2,3\}$. Hence $f_{[2]}=\{\{1,5\},\{3,4\},\{8,10\}\}$ and $f_{[3]}=\{\{2,6,11\},\{7,9,12\}\}$, so that $\Sfr_{[f]}\simeq\Sfr_3\times\Sfr_2$.
 
\subsubsection{Maximal subgroups of the tied symmetric monoid}

It was shown in \cite[Corollary 1]{AiArJu24} that $T\Sfr_n$ is an inverse monoid. Specifically, if $g=es$ for some $(e,s)\in P_n\times\Sfr_n$, then the inverse is given by $g^*=s^{-1}e$. See Figure~\ref{inverseElem}. This result, together with Proposition~\ref{TSn/R} and \cite[Corollary 3.5]{Steinberg2016}, provides an alternative proof that $\P_n$ is an inverse monoid.
\begin{figure}[H]\figeig\caption{An element in $T\Sfr_5$ and its inverse.}\label{inverseElem}\end{figure}

\begin{lemma}\label{lemMaxSubTS}
For each $e\in P_n\subset T\Sfr_n$, it holds 
$$G_e=\{es\mid s\in\Sfr_n\text{ and }se=es\}=eC_{\Sfr_n}(e),$$
where $C_{\Sfr_n}(e)$ is the centralizer of $e$ in the subgroup $\Sfr_n$ of $T\Sfr_n$.
\end{lemma}

\begin{proof}
Let $g\in T\Sfr_n$ such that $gg^*=e=g^*g$, and write $g=e's$ for some $(e',s)\in P_n\times\Sfr_n$. Then $g^*=s^{-1}e'$, and it follows that $e=gg^*=e'ss^{-1}e'=e'$. Also, $e=g^*g=s^{-1}ees=s^{-1}es$, hence $se=es$. Therefore $G_e=eC_{\Sfr_n}(e)$.
\end{proof}

By Lemma~\ref{lemMaxSubTS}, for any $e\in P_n$, the maximal subgroup at $e$ is given by $G_e=\{es\mid s\in\Sfr_n,\,ses^{-1}=e\}$, that is, $G_e$ consists of elements of the form $es$, where $s$ belongs to the stabilizer of $e$ under the conjugation action of $\Sfr_n$. This action realizes $T\Sfr_n$ as the semidirect product $P_n\rtimes\Sfr_n$.
 
For a subset $X\subseteq\dbrack{1,n}$, we denote by $\Sfr_X$ the subgroup of $\Sfr_n$ generated by the transpositions $s_{i,j}$ with $i,j\in X$ and $i<j$. Note that $\Sfr_X\simeq\Sfr_{|X|}$. More generally, given a set partition $I=(I_1,\ldots,I_k)\in P_n$, define the subgroup\[\Sfr_I=\Sfr_{I_1}\times\cdots\times\Sfr_{I_k}\simeq\Sfr_{|I_1|}\times\cdots\times\Sfr_{|I_k|}.\]
The subgroup $\Sfr_{[I]}$ in \eqref{maxSubParty} acts naturally on the subgroup $\Sfr_I$ by\[s\cdot s_{i,j}=s_{s(i),s(j)}\quad\text{for all}\quad s\in\Sfr_{[I]}.\]More precisely, let $A,B,X\in I$ with $|A|=|B|$ and $\min(A)<\min(B)$, and let $g:A\to B$ be the unique order-preserving bijection between $A$ and $B$. Then, for any $i,j\in X$ with $i<j$,\[s_{A,B}\cdot s_{i,j}=\begin{cases}
s_{g(i),g(j)}&\text{if }X=A\\
s_{g^{-1}(i),g^{-1}(j)}&\text{if }X=B\\
s_{i,j}&\text{otherwise}.
\end{cases}\]Note that the restriction of $s_{A,B}$ to $A$ is $g$, and its restriction to $B$ is $g^{-1}$. See Figure~\ref{maxTSaction}.
\begin{figure}[H]
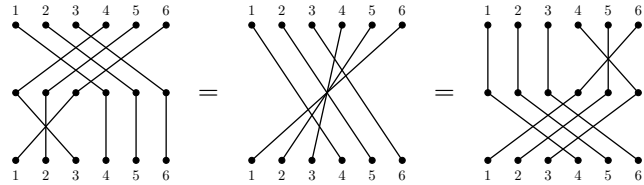
\figten\caption{ $s_{123|456}\cdot s_{1,3}=s_{4,6}\cdot s_{123|456}.$}\label{maxTSaction}\end{figure}

\begin{proposition}\label{ProMaxSubTSn}
The maximal subgroup of $T\Sfr_n$ at an idempotent $e\in P_n$ is $e(\Sfr_e\rtimes\Sfr_{[e]})$.
\end{proposition}

\begin{proof}
By definition, the group $\Sfr_e\rtimes\Sfr_{[e]}$ is generated by two kind of elements. First, the transpositions $s_{i,j}$ with $i,j$ belonging to the same block of $e$ and satisfying $i<j$, which commute with $e$. Second, the elements $s_{A,B}$ for all pair of blocks $A,B\in I$, which, by definition, stabilize the blocks of $e$, that is, $s_{A,B}e=es_{A,B}$. Consequently, by Lemma~\ref{lemMaxSubTS}, we have the inclusion $e(\Sfr_e\rtimes\Sfr_{[e]})\subset G_e$. Conversely, under the conjugation action, the set partition $e$ is stabilized by a permutation if and only if the permutation either acts within blocks or permutes blocks entirely, that is, it belongs to $\Sfr_e\rtimes\Sfr_{[e]}$. Therefore $G_e=e(\Sfr_e\rtimes\Sfr_{[e]})$.
\end{proof}
Proposition~\ref{ProMaxSubTSn} together with Lemma~\ref{lemMaxSubTS} imply that, for any $e\in P_n$, we have $C_{\Sfr_n}(e)=\Sfr_e\rtimes\Sfr_{[e]}$ and $G_e\simeq\Sfr_e\rtimes\Sfr_{[e]}$. For instance, as shown in Subsubsection~\ref{MaxGroupsParty}, if $e=(\{1,5\},\{2,6,11\},\{3,4\},\allowbreak\{7,9,12\},\allowbreak\{8,10\})$, then $\Sfr_{[e]}\simeq\Sfr_3\times\Sfr_2$. Therefore $G_e\simeq(\Sfr_2\times\Sfr_3\times\Sfr_2\times\Sfr_3\times\Sfr_2)\rtimes(\Sfr_3\times\Sfr_2)$. In particular, if $a_1,a_2,a_3\in\Sfr_2$ and $b_1,b_2\in\Sfr_3$, then $(s_2s_1,s_1)(a_1,b_1,a_2,b_2,a_3)=(a_3,b_2,a_1,b_1,a_2)\allowbreak(s_2s_1,s_1)$.

\subsubsection{Cellularity of the party algebra}
The cellularity of the party algebra was originally established in \cite{Kos08}. However, we present here an alternative and streamlined proof by exploiting the structure of $\P_n(\delta)$ as a twisted monoid algebra described in Subsubsection~\ref{twistedPartyAlgebraS}, alongside the characterization of the maximal subgroups of $\P_n$.

Recall that, for $g=fs$ and $g'=f's'$, $\tau_\beta(g,g')=\delta^{\beta(f,f')}$, where $\beta(f,f')$ counts the number of common standard arcs. Our approach relies on the framework developed by Wilcox \cite[Theorem 5]{Wilcox2007}, who showed that a twisted monoid algebra with a twisting taking invertible values is cellular, provided it satisfies five conditions, referred to by him as Assumptions 1--5. However, according to \cite[Corollary 6]{Wilcox2007}, Assumption 4 can be omitted.

We now examine Assumptions 1 and 2, which require the monoid to be equipped with an anti-involution with which the twisting is compatible. To address this, recall from Subsubsection~\ref{MaxGroupsParty} that the monoid $\P_n$ can be equipped with an anti-involution $*:\P_n\to\P_n$ defined by $(fs)^*=s^{-1}f$ for all $s\in\Sfr_n$ and $f\in P_n$. This compatibility is verified in the following lemma.
\begin{lemma}
For every $g,h\in\P_n$, we have $\tau_\beta(g,h)=\tau_\beta(h^*,g^*)$.
\end{lemma}
\begin{proof}
Let $g=se$ and $h=ft$ with $s,t\in\Sfr_n$ and $e,f\in P_n$. Then, the product $h^*g^*=t^{-1}fes^{-1}$ satisfies $\tau_\beta(h^*,g^*)=\beta(f,e)$. Since $\beta(f,e)=\beta(e,f)$, we conclude $\tau_\beta(h^*,g^*)=\beta(e,f)=\tau_\beta(g,h)$.
\end{proof}

Before continuing, we need to recall some facts about Green's relations. Two elements $g,h$ of a monoid $M$ are said to be \emph{$\L$-equivalent} (resp. \emph{$\R$-equivalent}, resp. \emph{$\J$-equivalent}), denoted by $g\equiv_\L h$ (resp. $g\equiv_\R h$, resp. $g\equiv_\J h$), if $Mg=Mh$ (resp. $gM=hM$, resp. $MgM=MhM$). The classes relative to these relations are called \emph{$\L$-classes}, \emph{$\R$-classes}, and \emph{$\J$-classes}; for $g \in M$, we denote the corresponding classes by $L_g$, $R_g$ and $J_g$, respectively. The equivalence relation generated by the union $\equiv_\L\cup\equiv_\R$ is denoted by $\equiv_\D$, and its classes, called \emph{$\D$-classes}, are known to be unions of $\L$-classes and unions of $\R$-classes \cite[Section 2.1]{How1995}.

Since $\P_n$ is a finite inverse monoid, several structural simplifications arise. First, the relations $\equiv_\D$ and $\equiv_\J$ coincide \cite[Proposition 2.1.4]{How1995}. Furthermore, every $\L$-class and every $\R$-class in $\P_n$ contains a unique idempotent \cite[Theorem 5.1.1]{How1995}, and consequently, each $\D$-class contains at least one idempotent \cite[Proposition 2.3.2]{How1995}.

For $g=fs\in\P_n$ with $f\in P_n$ and $s\in\Sfr_n$, we set $\|g\|=\|f\|$. The $\J$-classes of $\P_n$ are completely characterized by this value.

\begin{proposition}[{\cite[Proposition 3.5, and Proposition 3.9]{OrSaSchZaAl2022}}]\label{propJclassPn}
Two elements $g,h\in\P_n$ belong to the same $\J$-class if and only if $\|g\|=\|h\|$. Moreover, the $\J$-classes of $\P_n$ are in bijection with the partitions of $n$; specifically, for each $\lambda\in\Par_n$, the corresponding $\J$-class is $J_\lambda:=\{g\in\P_n\mid\|g\|=\lambda\}$. In addition, the $\L$-classes of $\P_n$ are in bijection with the set partitions of $\dbrack{1,n}$, where for each $f\in P_n$, the corresponding $\L$-class is $L_f=\{sf\mid s\in\Sfr_n\}$.
\end{proposition}

A set partition is said to be \emph{convex} if each of its blocks is an interval $\dbrack{a,b}$. For any $\lambda=(\lambda_1,\ldots,\lambda_k)\in\Par_n$, we denote by $\bar{\lambda}$ the unique convex set partition $\bar{\lambda}=(\bar{\lambda}_1,\ldots,\bar{\lambda}_k)\in P_n$ such that $|\bar{\lambda}_i|=\lambda_i$ for all $i\in\dbrack{1,k}$. By construction, $\|\bar{\lambda}\|=\lambda$, and therefore $J_{\bar{\lambda}}=J_\lambda$. Consequently, the set $\{\bar{\lambda}\mid\lambda\in\Par_n\}$ forms a system of idempotent representatives for the $\J$-classes of $\P_n$.

For Assumption 3 we require that each $\D$-class contains an idempotent fixed by the anti-involution. By Proposition~\ref{propJclassPn}, this condition is readily satisfied since $\bar{\lambda}^*=\bar{\lambda}$ for all $\lambda\in\Par_n$. Note that since $\P_n$ is equipped with an anti-involution $*$, the image of an $\L$-class is an $\R$-class; specifically, $L_g^* = R_{g^*}$ for any $g\in\P_n$. In the case of our canonical representatives, for $\lambda\in\Par_n$, the convex set partition $\bar{\lambda}$ satisfies $\bar{\lambda}^*=\bar{\lambda}$, which implies $L_{\bar{\lambda}}^*=R_{\bar{\lambda}}$. Thus, the maximal subgroup of $\P_n$ at $\bar{\lambda}$ is given by the intersection ${\Sfr_{[\bar{\lambda}]}}=L_{\bar{\lambda}}\cap L_{\bar{\lambda}}^*$ \cite[Exercise 1.19]{Steinberg2016}.

For Assumption 5, which requires the twisted group algebras of these maximal subgroups to be cellular, we have the following proposition. 
\begin{proposition}
Let $f\in P_n\subseteq\P_n$, and let $\Sfr_{[f]}\simeq\Sfr_{k_1}\times\cdots\times\Sfr_{k_q}$ with $\dbrace{f}=\{m_1<\cdots<m_q\}$ and $k_i=|I_{[m_i]}|$ for all $i\in\dbrack{1,q}$, be the maximal subgroup of $\P_n$ at $f$ as in \eqref{maxSubParty}. Then\[\CC^{\tau_\beta}[\Sfr_{[f]}]\simeq\CC[\Sfr_{k_1}]\otimes\cdots\otimes\CC[\Sfr_{k_q}].\]Consequently, the twisted algebra $\CC^{\tau_\beta}[\Sfr_{[f]}]$ is cellular.
\end{proposition}
\begin{proof}
By definition, any pair of elements $g,h\in\Sfr_{[f]}$ share the same underlying set partition, that is, $g=fs$ and $h=fs'$ for some $s,s'\in\Sfr_n$. Therefore, the twisting evaluates to a constant scalar $C:=\tau_\beta(g,h)=\delta^{\beta(f,f)}$ for all $g,h\in\Sfr_{[f]}$. Consider the linear isomorphism $\phi:\CC[\Sfr_{[f]}]\to\CC^{\tau_\beta}[\Sfr_{[f]}]$ given by $\phi(g)=C^{-1}g$ for all $g\in\Sfr_{[f]}$. For any $g,h\in\Sfr_{[f]}$, we have $\phi(g)\cdot\phi(h)=(C^{-1}g)\cdot(C^{-1}h)=C^{-2}\tau_\beta(g,h)gh=C^{-1}gh=\phi(gh)$. Thus, $\phi$ is an algebra isomorphism, which implies $\CC^{\tau_\beta}[\Sfr_{[f]}]\simeq\CC[\Sfr_{[f]}]\simeq\CC[\Sfr_{k_1}]\otimes\cdots\otimes\CC[\Sfr_{k_q}]$. Since the group algebra of the symmetric group is cellular and the tensor product of cellular algebras is also cellular \cite[Example (1.2)]{GrLe96} \cite[Subsection 3.2]{GeGo13}, the result follows.
\end{proof}

Finally, since the twisting $\tau_\beta$ takes values in $\CC^\times$, all requirements of the framework are satisfied. So, by \cite[Corollary~6]{Wilcox2007} we obtain the following theorem.
\begin{theorem}\label{Pn(d)cellular}
The party algebra $\P_n(\delta)$ is cellular.
\end{theorem}

\subsection{A twisted monoid algebra of the tied braid monoid}\label{SubSecTwisBT}

\subsubsection{}\label{twistedPnAlg}Restricting the twisting of $\P_n$ yields a twisting from $P_n$ into $\CC$, which we also denote by $\tau_\beta$. This gives rise to the twisted monoid algebra $P_n(\delta):=\CC^{\tau_\beta}[P_n]$, which we refer to as the \emph{twisted monoid algebra of set partitions}. Evidently, $P_n(\delta)$ coincides with the subalgebra of $\P_n(\delta)$ generated by $F_1,\ldots,F_{n-1}$, subject to the relations in \eqref{partyAlgRel2}.

\subsubsection{}\label{twistedTBn}The twisting from $P_n$ described above induces a twisting from $TB_n$ into $\CC$, which we also denote by $\tau_\beta$. Since $TB_n=P_n\rtimes B_n$, for each pair of elements $g,g'\in TB_n$ there are unique permutations $s,s'\in\Sfr_n$ and unique set partitions $e,e'\in P_n$ such that $g=se$ and $g'=e's'$. For $\delta\in\CC^\times$, we then set $\tau_\beta(g,g')=\delta^{\beta(e,e')}$, where $\beta(e,e')$ is defined as in Subsubsection~\ref{twistedPnAlg}.

The algebra $TB_n(\delta):=\CC^{\tau_\beta}[TB_n]$ will be called the \emph{twisted monoid algebra of tied braids}. By applying Theorem~\ref{TwistedPresentation}, we obtain that $\CC^{\tau_\beta}[TB_n]$ is presented by generators $\sigma_1,\ldots,\sigma_{n-1}$ satisfying the braid relations in \eqref{braidRels}, generators $\sigma_1^{-1},\ldots,\sigma_{n-1}^{-1}$ satisfying \eqref{TBM1}, and generators $\bar{e}_1,\ldots,\bar{e}_{n-1}$, subject to the following relations:\begin{gather}
\bar{e}_i^2=\delta\bar{e}_i,\qquad \bar{e}_i\bar{e}_j=\bar{e}_j\bar{e}_i,\\
\sigma_i\bar{e}_j=\bar{e}_j\sigma_i\quad\text{if }|i-j|\neq1,\qquad\sigma_i\sigma_j^{\pm1}\bar{e}_i=\bar{e}_j\sigma_i\sigma_j^{\pm1}\quad\text{if }|i-j|=1,\\
\sigma_i\bar{e}_j\bar{e}_i=\bar{e}_j\sigma_i\bar{e}_j=\bar{e}_i\bar{e}_j\sigma_i\quad\text{if }|i-j|=1.
\end{gather}

\subsection{Party-Brauer-like monoids}\label{SubSecParBrMon}
\subsubsection{}
The \emph{$2$-modular party monoid} or \emph{$2$-tonal partition monoid} is the submonoid $\P^{_{(2)}}_n$ of $\Cfr_n$ generated by $s_1,\ldots,s_{n-1}$, $t_1,\ldots,t_{n-1}$ and $f_1,\ldots,f_{n-1}$ \cite{Ta1997}. Its cardinality $|\P^{_{(2)}}_n|$ is given by the sequence \cite[\href{https://oeis.org/A005046}{A005046}]{OEIS} \cite[Subsection 6.2]{AhMaMa21}. With these generators, the monoid $\P^{_{(2)}}_n$ admits a presentation by relations \eqref{TSn1}, \eqref{Br1}, \eqref{Br2}, \eqref{Par1}, \eqref{Par2}, \eqref{Par3}, together with the following additional ones \cite[Subsection 3.3]{Ore05}:\begin{gather}
t_if_i=t_i,\qquad t_if_j=f_jt_i\quad\text{if }|i-j|\neq1,\qquad f_it_jf_i=f_jf_i\quad\text{if }|i-j|=1,\label{PBr1}\\
t_if_jt_i=t_i\quad\text{if }|i-j|=1.\label{PBr2}
\end{gather}

\begin{proposition}
The monoid $\P^{_{(2)}}_n$ is the quotient of the ramified Brauer monoid $\R\Brfr_n$ by the relations $e_i=d_i$ for all $i$.
\end{proposition}
\begin{proof}
First, observe that \eqref{PBr2} is superfluous. Indeed, by applying \eqref{Br1}, \eqref{PBr1} and \eqref{PBr2}, respectively, we obtain\[t_i=t_it_jt_i=t_if_it_jf_it_i=t_if_jf_it_i=t_if_jt_i.\]
Now, setting $e_i=d_i=f_i$ in the presentation of $\R\Brfr_n$ given in Subsubsection~\ref{BrauerSSS}, we immediately obtain that $\P^{_{(2)}}_n$ is presented by generators $s_1,\dots,s_{n-1}$, $t_1,\ldots,t_{n-1}$, $f_1,\ldots,f_{n-1}$ subject to \eqref{TSn1}, \eqref{Br1}, \eqref{Br2}, \eqref{Par1}, \eqref{Par2}, \eqref{Par3}, \eqref{PBr1}, and the relations\[s_if_jf_i=f_js_if_j=f_if_js_i\quad\text{if }|i-j|=1.\]However, as shown in the proof of Proposition~\ref{TSn/R}, this relation is a consequence of \eqref{Par1}, \eqref{Par2} and \eqref{Par3}. This completes the proof.
\end{proof}

\begin{remark}
Proposition \ref{TSn/R} motivates investigating the monoid obtained by adding the relation $s_ie_i=e_i$ to the presentation of $\R\Brfr_n$ given in Subsubsection~\ref{BrauerSSS}.
\end{remark}

\section{Party-Hecke algebras}\label{SecParty-Hecke}

In this section, we introduce the Party-Hecke algebra (Definition \ref{Pn(p,q)}) and prove that it can be realized as a quotient of a twisted monoid algebra of the tied braid group. In Theorem~\ref{TR}, we construct a tensorial representation of this algebra, which we then use to establish a linear basis. Finally, Theorem~\ref{GeSS} shows that the Party-Hecke algebra is generically semisimple.

\subsection{The Party-Hecke algebra}\label{SubsecParty-Hecke}

In what follows, we set $p,q\in\CC^\times$.

\subsubsection{}\label{SubsubsecParty-Hecke}{We begin our construction of the Party-Hecke algebra by introducing a new monoid, referred to as the \emph{party braid monoid}.

\begin{definition}\label{BPn}
The \emph{party braid monoid} $\BP_n$ is the monoid presented by generators $\sigma_1,\ldots,\sigma_{n-1}$ satisfying the braid relations in \eqref{braidRels}, and commuting generators $\bar{f}_1,\ldots,\bar{f}_{n-1}$, together with the following mixed relations:
\begin{gather}
\sigma_i \bar{f}_j = \bar{f}_j \sigma_i\quad\text{if }|i-j|\neq1,\qquad\sigma_i\sigma_j \bar{f}_i=\bar{f}_j\sigma_i\sigma_j\quad\text{if }|i-j|=1.\label{BP2}
\end{gather}
Observe that $\BP_n$ contains a copy of the braid monoid $B_n^+$. Also, note that the party monoid $\P_n$ is a quotient of the party braid monoid $\BP_n$.
\end{definition}

\subsubsection{}
{We introduce the Party-Hecke algebra as a quotient of the monoid algebra of the party braid monoid.}

\begin{definition}\label{Pn(p,q)}
The \emph{Party-Hecke algebra} $\Pa_n(p,q)$ is defined as the quotient of the monoid algebra $\CC[\BP_n]$ by the two-sided ideal generated by the following elements:
\begin{equation}\label{PH1}
\sigma_i^2-pq^2-p(p-1)\bar{f}_i,\qquad\sigma_i\bar{f}_i-pq\bar{f}_i,\qquad\bar{f}_i^2-q^2\bar{f}_i
\end{equation}
\end{definition}
We denote by $G_i$ the image of $\sigma_i$ in the quotient and by $F_i$ the image of $\bar{f}_i$. Hence, $\Pa_n(p,q)$ is the $\CC$-algebra presented by generators $G_1,\ldots,G_{n-1}$ and $F_1,\ldots,F_{n-1}$, subject to the relations:
\begin{align}
G_i^2=pq^2+p(p-1)F_i,&\qquad G_iF_i=F_iG_i=pqF_i,\label{G^2}\\
G_iG_j=G_jG_i\qquad\text{if }|i-j|>1,&\qquad G_iG_jG_i=G_jG_jG_i\qquad\text{if }|i-j|=1,\label{P-H1}\\
F_i^2=q^2F_i,&\qquad F_iF_j=F_jF_i,\label{P-H2}\\
G_iF_j=F_jG_i\quad\text{if }|i-j|\neq1,&\qquad G_iG_jF_i=F_jG_iG_j,\quad\text{if }|i-j|=1.\label{P-H3}
\end{align}
Relation \eqref{G^2} implies that each $G_i$ is invertible with inverse given by
\begin{gather}G_i^{-1}=p^{-1}q^{-2}G_i+q^{-3}(p^{-1}-1)F_i.\label{G^-1}\end{gather}Observe that $\Pa_n(1,1)=\CC[\P_n]$, the monoid algebra of the party monoid.

\subsection{The Party-Hecke algebra as a quotient of $TB_n(q^2)$}\label{SubsecParty-HeckeQuot}
We now show that the algebra $\Pa_n(p,q)$ is a quotient of $TB_n(q^2)$. To establish this, we first require the following lemma.
\begin{lemma}\label{BTRels}
For each $i,j\in\dbrack{1,n-1}$ with $|i-j|=1$, we have:
\begin{enumerate}
\item $G_iF_jF_i=F_jG_iF_j=F_iF_jG_i$.\label{TTD}
\item $G_iG_j^{-1}F_i=F_jG_iG_j^{-1}$.\label{BP2^-1}
\end{enumerate}
\end{lemma}
\begin{proof}
First, we show \eqref{TTD} as follows:\[\begin{array}{rcl}
\!F_jG_iF_j&\!\stackrel{\eqref{P-H3}}{=}\!&G_iG_jF_iG_j^{-1}F_j\\
&\!\stackrel{\eqref{G^2}}{=}\!&w^{-1}G_iG_jF_iF_j\ \stackrel{\eqref{G^2}}{=}\ G_iF_jF_i\ \stackrel{\eqref{G^2}}{=}\ wF_iF_j\ \stackrel{\eqref{G^2}}{=}\ F_iF_jG_i.
\end{array}\]Then, by applying \eqref{TTD}, we get \eqref{BP2^-1}:\[\begin{array}{rcl}
G_iG_j^{-1}F_i&\stackrel{\eqref{G^-1}}{=}&G_i(p^{-1}q^{-2}G_j+q^{-3}(p^{-1}-1)F_j)F_i\\[0.15cm]
&=&p^{-1}q^{-2}G_iG_jF_i+q^{-3}(p^{-1}-1)G_iF_jF_i\\
&\stackrel{\eqref{P-H3}}{=}&p^{-1}q^{-2}F_jG_iG_j+q^{-3}(p^{-1}-1)F_jG_iF_j\\[0.15cm]
&=&F_jG_i(p^{-1}q^{-2}G_j+q^{-3}(p^{-1}-1)F_j)\\
&\stackrel{\eqref{G^-1}}{=}&F_jG_iG_j^{-1}.
\end{array}\]
\end{proof}

\begin{proposition}\label{TBnQuotient}
The map sending $\sigma_i\mapsto G_i$ and $\bar{e}_i\mapsto F_i$ defines an algebra homomorphism $TB_n(q^2)\to\Pa_n(p,q)$. More precisely, the Party-Hecke algebra $\Pa_n(p,q)$ is the quotient of the twisted monoid algebra $TB_n(q^2)$ by the two-sided ideal generated by the elements $\sigma_i^2-pq^2-p(p-1)\bar{e}_i$ and $\sigma_i\bar{e}_i-pq\bar{e}_i$ for all $i\in\dbrack{1,n-1}$.
\end{proposition}
\begin{proof}
This follows directly from Eq.~\eqref{G^-1}, together with Lemma~\ref{BTRels} and the presentation of $TB_n(q^2)$ given in Subsubsection~\ref{twistedTBn}.
\end{proof}

\subsection{A linear basis}\label{SubsecBasisParty-Hecke}

The goal of this subsection is to give a linear basis $\G_n$ for $\Pa_n(p,q)$. To explain the elements of $\G_n$, we need the following notations.

Given $s\in\Sfr_n$, we set $G_s: =G_{i_1}\cdots G_{i_k}\in\Pa_n(p,q)$, where $s_{i_1}\cdots s_{i_k}$ is a reduced expression of $s$. In particular, $G_{s_j}= G_j$. Now, relations in \eqref{P-H1} and Matsumoto's theorem \cite{Matsumoto1964} imply that $G_s$ does not depend on the choice of the reduced expression. 
Now, for each pair of indexes $i,j\in\dbrack{1,n-1}$ with $i<j$, we set\begin{gather*}
F_{i,j}=F_{j,i}=G_i\cdots G_{j-2}F_{j-1}G_{j-2}^{-1}\cdots G_i^{-1}.
\end{gather*}

\begin{corollary}\label{RHLemma2&3}
The map sending $f_{i,j}\mapsto F_{i,j}$ defines a monoid homomorphism from $P_n$ to $\Pa_n(p,q)$. In consequence, each $F_{i,j}$ is a commuting idempotent satisfying $G_iF_{j,k}=F_{s_i(j),s_i(k)}G_i$ for all $i,j,k\in\dbrack{1,n}$ with $j<k$.
\end{corollary}
\begin{proof}
This is a direct consequence of Proposition~\ref{TBnQuotient}, Equation \eqref{eijTBn}, and the fact that $P_n$ embeds into $TB_n$ by means of the map $f_{i,j}\mapsto e_{i,j}$.
\end{proof}
As in \cite[Section 3]{RyomHansen}, for each nonempty subset $B\subseteq\dbrack{1,n}$ and for each set partition $I=\{I_1,\ldots,I_k\}$ of $\dbrack{1,n}$, we set\[F_B=\prod_{i,j\in B,\,i<j}F_{i,j}\quad\text{and}\quad F_I=F_{I_1}\cdots F_{I_k},\]where $F_B=1$ whenever $|B|=1$.

\begin{proposition}\label{PREspanP-H}
The set of products $F_IG_s$, with $I\in P_n$ and $s\in\Sfr_n$, spans $\Pa_n(p,q)$.
\end{proposition}
\begin{proof}
This is a consequence of the quadratic relation in \eqref{G^2}, Proposition~\ref{TBnQuotient} and the fact that $TB_n$ is the semi direct product $P_n\rtimes B_n$.
\end{proof}

Define\[\G_n:=\{F_IG_s\mid I\in P_n\text{ and }s\in\Sfr_n\text{ with }(I,s)\text{ coprime}\}.\]Observe that, by virtue of Proposition~\ref{partyNF}, the elements of $\G_n$ are parametrized by $\P_n$.

To establish that $\G_n$ spans $\Pa_n(p,q)$, we first need to introduce some additional notations and three technical lemmas.

For each $i,j\in\dbrack{1,n}$ with $i<j$, we set\begin{gather*}
G_{i,j}=G_{j,i}=G_i\cdots G_{j-2}G_{j-1}G_{j-2}^{-1}\cdots G_i^{-1}.
\end{gather*}
Observe that, as generators $G_i$ satisfy the braid relations, then, due to \eqref{genSigma}, the map $\sigma_{i,j}\mapsto G_{i,j}$ defines a group homomorphism from $B_n$ to $\Pa_n(p,q)$.

\begin{lemma}\label{DualGi}
For each $i,j,k,l$, we have:
\begin{enumerate}
\item $G_{i,j}^2=pq^2+p(p-1)F_{i,j}$.\label{DualGi1}
\item $G_{i,j}F_{k,l}=F_{s_{i,j}(k),s_{i,j}(l)}G_{i,j}$.\label{DualGi2}
\item $G_{i,j}F_{i,j}=F_{i,j}G_{i,j}=pqF_{i,j}$.\label{DualGi3}
\item $G_{i,j}^{-1}=p^{-1}q^{-2}G_{i,j}+q^{-3}(p^{-1}-1)F_{i,j}$.\label{DualGi4}
\end{enumerate}
\end{lemma}
\begin{proof}
We show \eqref{DualGi1} as follows
\[\begin{array}{rclll}
G_{i,j}^2&\!\!\!=\!\!\!&G_i\cdots G_{j-2}G_{j-1}^2G_{j-2}^{-1}\cdots G_i^{-1}\\[0.15cm]
&\!\!\!\stackrel{\eqref{G^2}}{=}\!\!\!&G_i\cdots G_{j-2}(pq^2+p(p-1)F_{j-1})G_{j-2}^{-1}\cdots G_i^{-1}&\!\!=\!\!&pq^2+p(p-1)F_{i,j}.
\end{array}\]We get \eqref{DualGi2} and \eqref{DualGi3} as direct consequences of \eqref{G^2} and Corollary~\ref{RHLemma2&3}. Claim \eqref{DualGi4} is a direct consequence of claims \eqref{DualGi1} and \eqref{DualGi3}.
\end{proof}

\begin{lemma}\label{spanLem1}
Let $i,j,k\in\dbrack{1,n}$ and let $\tau\in\Sfr_n$ such that $s_{i,j}=\tau s_k\tau^{-1}$. Then $F_{i,j}G_\tau G_kG_\tau^{-1}=pqF_{i,j}$.
\end{lemma}
\begin{proof}
Since $\tau^{-1}s_{i,j}\tau=s_k$, then, by Lemma~\ref{DualGi}\eqref{DualGi2} and Lemma~\ref{DualGi}\eqref{DualGi3}, we get the result as follows:\[\begin{array}{rcl}
F_{i,j}G_\tau G_kG_\tau^{-1}&=&G_\tau(G_\tau^{-1}F_{i,j}G_\tau) G_kG_\tau^{-1}\\
&=&G_\tau F_{\tau^{-1}(i),\tau^{-1}(j)}G_kG_\tau^{-1}\\
&=&G_\tau F_kG_kG_\tau^{-1}\\
&=&pqG_\tau F_kG_\tau^{-1}\,\,\,=\,\,\,pqF_{i,j}.
\end{array}\]
\end{proof}

\begin{lemma}\label{spanLem2}
Let $\nu\in\Sfr_n$ and let $(i,j)$ be an inversion of $\nu$. Then, there is $\tau\in\Sfr_n$ and $k\in\dbrack{1,n-1}$, such that $G_\nu=G_\tau G_kG_\tau^{-1}G_{s_{i,j}\nu}$ with $s_{i,j}=\tau s_k\tau^{-1}$.
\end{lemma}
\begin{proof}
Let $r=\ell(\nu)$, that is, $\nu=s_{i_1}\cdots s_{i_r}$ for some reduced expression. Then, due to the strong exchange condition \cite[Subsection 5.8]{Humph1997}, there is an $m\in\dbrack{1,r}$ such that $s_{i,j}\nu=s_{i_1}\cdots s_{i_{m-1}}s_{i_{m+1}}\cdots s_{i_r}$. Let $\tau=s_{i_1}\cdots s_{i_{m-1}}$, $k=i_m$, and $\epsilon=s_{i_{m+1}}\cdots s_{i_r}$. Thus, we get\[\nu=\tau s_k\epsilon,\qquad s_{i,j}\nu=\tau\epsilon,\qquad s_{i,j}=(s_{i,j}\nu)\nu^{-1}=\tau\epsilon(\tau s_k\epsilon)^{-1}=\tau s_k\tau^{-1}.\]Therefore $G_\nu=G_\tau G_kG_\epsilon=G_\tau G_kG_\tau^{-1}G_\tau G_\epsilon=G_\tau G_kG_\tau^{-1}G_{s_{i,j}\nu}$.
\end{proof}

\begin{proposition}\label{spanP-H}
The set $\G_n$ spans $\Pa_n(p,q)$.
\end{proposition}
\begin{proof}
Due to Proposition~\ref{PREspanP-H}, it is enough to show that each product $F_IG_\nu$ can be written as $\alpha F_IG_{\nu'}$ for some $\alpha\in\CC$, where the pair $(I,\nu')$ is coprime.

Let $(i,j)$ be an inversion of $\nu$ satisfying $I=If_{i,j}$. Due to Corollary~\ref{RHLemma2&3} and Lemma~\ref{spanLem2}, we have $F_IG_\nu=F_IF_{i,j}G_\tau G_kG_\tau^{-1}G_{s_{i,j}\nu}$ for some $k\in\dbrack{1,n-1}$ and $\tau\in\Sfr_n$ such that $s_{i,j}=\tau s_k\tau^{-1}$. Then, by applying Lemma~\ref{spanLem1}, we get $F_IG_\nu=wF_IF_{i,j}G_{s_{i,j}\nu}=wF_IG_{s_{i,j}\nu}$. Since $\ell(s_{i,j}\nu)<\ell(\nu)$, we can repeat the process for all inversions of $\nu$ occurring in $I$, until the length of the resulting permutation is minimal. This proves the proposition.
\end{proof}

\begin{remark}\label{coprimeBElems}
By Proposition~\ref{partyNF}, if $(I,v)$ and $(J,w)$ are two distinct coprime pairs, then the products $G_vF_I$ and $G_wF_J$ are necessarily distinct. In other words, the coprimality condition ensures that each element in $\G_n$ corresponds to a unique coprime pair, and thus no two different pairs can yield the same product. This guarantees that $\G_n$ consists of distinct elements.
\end{remark}

Next, we introduce a deformation of Tanabe's tensor representation for the party algebra. See \cite[Subsection 2.1]{Ta1997} and compare with \cite[Proposition 3.1]{AgOr2008}. This representation will be used to show that $\G_n$ is linearly independent.

Let $m$ be a positive integer, and let $V$ be the $\CC$-vector space spanned by $B=\{v_i^r\mid i,r\in\dbrack{1,m}\}$. As usual, we denote by $B^{\otimes n}$ the standard basis of $V^{\otimes n}$ associated to $B$, consisting of the vectors:\[v_I^R:=v_{i_1}^{r_1}\otimes\cdots\otimes v_{i_n}^{r_n},\]
where $I=(i_1,\ldots,i_n)$ and $R=(r_1,\ldots,r_n)\in\dbrack{1,m}^n$. When there is no risk of confusion, we simply write $v_{i_1\ldots i_n}^{r_1\ldots r_n}$ instead of $v_{(i_1,\ldots,i_n)}^{(r_1,\ldots,r_n)}$.

Define $\wF,\wG\in\End(V^{\otimes2})$ by
\[\wF(v_i^r\otimes v_j^s)=
\begin{cases}
q^{2}v_i^r\otimes v_j^s&\text{if }r=s\\
0&\text{if }r\neq s,
\end{cases}\qquad\wG(v_i^r\otimes v_j^s)=
\begin{cases}
pq\,v_j^s\otimes v_i^r&\text{if }r=s\\
pq\,v_j^s\otimes v_i^r&\text{if }r\neq s,\,i>j\\
q\sqrt{ p}\,v_j^s\otimes v_i^r&\text{if }r\neq s,\,i=j\\
qv_j^s\otimes v_i^r&\text{if }r\neq s,\,i<j.
\end{cases}\]
These operators satisfy the following Party-Hecke relations:
\begin{equation}\label{FGtilde}
\wF^{\,2}=q^{2}\wF,\qquad\wG\wF=\wF\wG=pq\wF,\qquad\wG^{\,2}=pq^{2}+p(p-1)\wF.\end{equation}

For $i\in\dbrack{1,n-1}$, define $\wF_i$ (resp. $\wG_i$) as the endomorphism of $V^{\otimes n}$ acting by $\wF$ (resp. $\wG$) on the tensor factors in positions $(i,i+1)$, and as the identity otherwise. Evidently, we have:
\begin{equation}\label{ijkltilde}
\wF_k\wF_l=\wF_l\wF_k,\quad\wF_i\wG_j=\wG_j\wF_i\quad\text{and}\quad\wG_i\wG_j=\wG_j\wG_i\quad\text{for $|i-j|>1$.}
\end{equation}

\begin{remark}\label{UpToS}
Since each $\wG_i$ acts, up to a scalar, as the elementary transposition $s_i$, it follows that there exist scalars $\lambda_{I,R}\in\CC$ such that\[(\wG_{i_1}\cdots\wG_{i_j})(v_I^R)=\lambda_{I,R}v_{s(I)}^{s(R)},\quad\text{where}\quad s:=s_{i_1}\cdots s_{i_j}.\]
\end{remark}

\begin{theorem}\label{TR}
The mapping $G_i\mapsto\wG_i$, $F_i\mapsto \wF_i$ defines a representation $\psi_{p,q}$ of $\Pa_n(p,q)$ in $V^{\otimes n}$. In particular, $\psi_{1,q}$ is the Tanabe representation.
\end{theorem}
\begin{proof}
Without loss of generality, we can assume $n = 3$. It suffices to show that the assignment respects the defining relations of $\Pa_n( p,q)$, namely \eqref{G^2}--\eqref{P-H3}.

From \eqref{FGtilde} we deduce that relations \eqref{G^2} and \eqref{P-H2} hold for the operators $\wG_i$'s and $\wF_i$'s. Together with \eqref{ijkltilde}, this shows that the only relations left to be verified are the braid relation in \eqref{P-H1} and the second relation in \eqref{P-H3}.

We begin with the second relation in \eqref{P-H3}. It is enough to prove that, for every basis element $v_{ijk}^{rst}$ of $V^{\otimes3}$, one has:\begin{equation}\label{EqTR}\big(\wF_1\wG_2\wG_1\big)(v_{ijk}^{rst})=\big(\wG_2\wG_1\wF_2\big)(v_{ijk}^{rst}).\end{equation}
If all upper indices of $v_{ijk}^{rst}$ are equal, that is, $r=s=t$, then both $\wF_1$ and $\wF_2$ act as scalar multiplication by $q^2$, and so \eqref{EqTR} is immediately satisfied. If instead all upper indices are distinct, both $\wF_1$ and $\wF_2$ act as zero, and again \eqref{EqTR} holds. Thus, we only need to consider the three remaining cases: $r=s\neq t$, $r\neq s=t$ and $r=t\neq s$.

In the first case we compute\[\big(\wF_1\wG_2\wG_1\big)(v_{ijk}^{rrt})=\big(\wF_1\wG_2\big)(pqv_{jik}^{rrt})=pq\wF_1\big(\wG_2(v_{jik}^{rrt})\big)=0,\]
 since $\wG_2(v_{jik}^{rrt})$ has upper indices of type $rtr$, on which $\tilde{F}_1$ acts as zero. On the other hand,\[\big(\wG_2\wG_1\wF_2\big)(v_{ijk}^{rrt})=\big(\wG_2\wG_1\big)(0)=0.\]Therefore \eqref{EqTR} holds in this case. The other two cases are analogous. We conclude that the second relation in \eqref{P-H3} is satisfied.

It remains to check the braid relation in \eqref{P-H1}, that is,\[\big(\wG_1\wG_2\wG_1\big)(v_{ijk}^{rst})=\big(\wG_2\wG_1\wG_2\big)(v_{ijk}^{rst}).\]We proceed by distinguishing cases according to the upper indices $(r,s,t)$: (i) $r=s=t$, (ii) $r=s\neq t$, (iii) $r\neq s=t$, (iv) $r=t\neq s$, and (v) $r,s,t$ all distinct.

\underline{Case (i)}. A direct calculation gives $(G_1G_2G_1)(v_{ijk}^{rrr})=(G_2G_1G_2)(v_{ijk}^{rrr})=p^3q^3\,v_{kji}^{rrr}$.

\underline{Case (ii)}. We subdivide into five types of lower indices $(i,j,k)$ into the following five types: (a) $i=j=k$, (b) $i\neq j\neq k\neq i$, (c) $i=j\neq k$, $(d)$ $i\neq j=k$, and $(e)$ $i=k\neq j$.

For type (a), one obtains
\begin{align*}
(G_1G_2G_1)(v_{iii}^{rrt})&=(G_1G_2)(pq\,v_{iii}^{rrt})=pq^2\sqrt{p}\,G_1(v_{iii}^{rtr})=q^3p^2\,v_{iii}^{trr},\\
(G_2G_1G_2)(v_{iii}^{rrt})&=(G_2G_1)(q\sqrt{p}\,v_{iii}^{rtr})=q^2p G_2(v_{iii}^{trr})=q^3p^2 v_{iii}^{trr}.
\end{align*}
For type (b), we only analyze the most representative case, that is, $i<j, j>k$, and $i<k$:
\begin{align*}
(G_1G_2G_1)(v_{ijk}^{rrt})&=(G_1G_2)(pq\, v_{jik}^{rrt}) = p\sqrt{p}q^2 \, G_1 ( v_{jki}^{rtr})= p^2q^3\,v_{kji}^{trr} ,\\
(G_2G_1G_2)(v_{ijk}^{rrt})&=(G_2G_1)(q\sqrt{p}\,v_{ikj}^{rtr}) = pq^2 G_2( v_{kij}^{trr})
= p^2q^3v_{kji}^{trr}.
\end{align*}
 For type (c), we will first analyze the possibility $i=j<k$:
\begin{align*}
(G_1G_2G_1)(v_{iik}^{rrt})& =pq\,(G_1G_2)(v_{iik}^{rrt}) = pq^2 \, G_1 ( v_{iki}^{rtr})= pq^3\,v_{
kii}^{trr} ,\\
(G_2G_1G_2)(v_{iik}^{rrt})& =q(G_2G_1)(v_{iki}^{rtr}) =q^2 G_2( v_{kii}^{trr})
= pq^3v_{kii}^{trr}.
\end{align*}
The other possibility, $i=j>k$, is again a direct computation.

Cases (iii) (iv) are both analogous to case (ii).

\underline{Case (v)}. Here all five types (a) to (e) of lower indices occur. Most are straightforward, except for the subcase $i<j>k<i$ in (b), which we compute in detail:
\begin{align*}
(G_1G_2G_1)(v_{ijk}^{rst}) = q(G_1G_2)(v_{jik}^{srt}) = pq^2 \, G_1 ( v_{jki}^{str})= p^2q^3\,v_{
kji}^{tsr} ,\\
(G_2G_1G_2)(v_{ijk}^{rst}) =pq\,(G_2G_1)(v_{ikj}^{rts}) = p^2q^2\,G_2( v_{kij}^{trs})
= p^2q^3v_{kji}^{tsr}.
\end{align*}
Thus, the relation holds in the most delicate situation, and therefore in all others as well.

Hence, the assignment $G_i\mapsto\wG_i$, $F_i\mapsto \wF_i$ defines a representation of $\Pa_n(p,q)$.
\end{proof}

\begin{theorem}\label{Gnbasis}
For $n\leq m^2$, the representation $\psi_{p,q}$ is faithful. Consequently, $\G_n$ is a linear basis for $\Pa_n(p,q)$, and so $\dim\Pa_n(p,q)=|\P_n|$.
\end{theorem}
\begin{proof}The proof proceeds by showing that $\G_n$ is a linear basis and that its image under $\psi_{p,q}$ remains linearly independent. By Proposition~\ref{spanP-H}, to prove that $\G_n$ is a basis, it suffices to show that $\G_n$ is linearly independent. Set $V$ to be the $\CC$-vector space with basis $B=\{v_i^r\mid i,r\in\dbrack{1,m}\}$, where $n\leq m^2$. Then $V^{\otimes n}$ has a basis $B^{\otimes n}$ consisting of vectors of the form:\[v_I^R:=v_{i_1}^{r_1}\otimes\cdots\otimes v_{i_n}^{r_n}.\]Suppose we have the linear combination:\[\sum_{(I,w)\atop\text{coprime}}\lambda_{(I,w)}G_w F_I=0.\]Then, applying $\psi_{p,q}$, we obtain
\begin{equation}\label{LinCom}
\sum_{(I,w)\atop\text{coprime}}\lambda_{(I,w)}\wG_w\wF_I=0.
\end{equation}
Put $v:= v'\otimes v''$, where $v'=v_1^1\otimes\cdots\otimes v_n^1$, $v''=v_m^m\otimes\cdots\otimes v_m^m$ with $m^2-n$ factors. Evaluating the linear combination \eqref{LinCom} at $v$, and observing that $\psi_{p,q}$ acts trivially on $v''$, we obtain:
\[0=\sum_{(I,w)\atop\text{coprime}}\lambda_{(I,w)}\wG_w\wF_I(v)\otimes v''=\sum_{(I,w)\atop\text{coprime}}\lambda_{(I,w)}a_I\wG_w(v)\otimes v'',\]where each $a_I$ is certain positive integer. Notice that the vectors $\wG_w(v)\otimes v''$ are elements of the basis $B^{\otimes n}$. Therefore, using Remark~\ref{coprimeBElems} and Remark~\ref{UpToS}, it follows that all coefficients $\lambda_{(I,w)}a_I$ must be $0$. Since each $a_I$ is positive, we conclude that $\lambda_{(I,w)}=0$ for all $(I,w)$.
\end{proof}

\subsection{Generic representations}\label{SubsecBasisGenRep}

As noted in \cite[Subsection~3.5]{OrSaSchZaAl2022}, by \cite[Corollary 9.4]{Steinberg2016}, the algebra $\CC[\P_n]=\Pa_n(1,1)$ is semisimple. Moreover, since $\CC$ is an algebraically closed field, applying the Wedderburn--Artin theorem (cf. \cite[Chapter 5]{ErdmannHolm2018}) yields that $\Pa_n(1,1)$ is split semisimple.

\begin{theorem}\label{GeSS}
The Party-Hecke algebra $\Pa_n(p,q)$ is generically semisimple; that is, there exists a nonempty Zariski open subset of $\CC^2$ for which $\Pa_n(p,q)$ is semisimple. 
\end{theorem}
\begin{proof}
By \cite[Lemma 1.6]{CliParSco1999}, the property of being split semisimple is generic; that is, it holds on a nonempty Zariski open subset of the parameter space. As explained in \cite[Section 1]{CliParSco1999}, since $\Pa_n(p,q)$ is a free module of finite rank over $\CC[p,q]$ (Theorem~\ref{Gnbasis}), and since the specialization $\Pa_n(1,1)$ is split semisimple, it follows that there exists a nonempty Zariski open subset $U$ of $\CC^2$ such that $\Pa_n(p,q)$ is split semisimple for all $(p,q)\in U$. In particular, since $\CC$ is algebraically closed, then $\Pa_n(p,q)$ is generically semisimple.
\end{proof}

\section{The Party-Hecke algebra as a quotient of the bt-algebra and related quotients}\label{SecParty-HeckeQuo}

In this section, we establish that the Party-Hecke algebra is a quotient of the bt-algebra (Subsection~\ref{P'asQuot}). Furthermore, we determine two natural quotients of Temperley--Lieb type (Subsection~\ref{TwoTL}), and introduce a partition Temperley--Lieb-like quotient (Subsection~\ref{QuoPTL}).

\subsection{$\Pa_n'(p,q)$ as a quotient of $\E_n(\sqrt{p})$}\label{P'asQuot}

The rescaling $H_i=(q\sqrt{p})^{\,-1}G_i$ yields another presentation of the Party-Hecke algebra. Precisely, consider the generators $F_1,\ldots,F_{n-1}$ satisfying \eqref{P-H2}, together with generators $H_1,\ldots,H_{n-1}$ satisfying the same defining relations as in $\Pa_n(p,q)$, except that the relations in \eqref{G^2} are replaced, that is, \begin{align}
H_i^2=1+q^{-2}(p-1)F_i,&\quad H_iF_i=F_iH_i=\sqrt{p}F_i,\label{H^2}\\
 H_iH_j=H_jH_i\quad\text{if }|i-j|>1,&\quad H_iH_jH_i=H_jH_jH_i\quad\text{if }|i-j|=1,\label{HRels2}\\
 H_iF_j=F_jH_i\quad\text{if }|i-j|\neq1, &\quad H_iH_jF_i=F_jH_iH_j,\quad\text{if }|i-j|=1. \label{HRels3}
\end{align}
We denote by $\Pa_n'(p,q)$ the Party-Hecke algebra with the above presentation. Note that:\begin{equation}H_i^{-1}=H_i-q^{-2}(\sqrt{p}-\sqrt{p}^{\,-1})F_i.\label{H^-1}\end{equation}

By Proposition~\ref{TSn/R}, it is desirable to define a deformation of $\CC[\P_n]$ as a quotient of $\E_n(u,v)$. Owing to the definition of the congruence $R$, such a deformation can be obtained by deforming the defining generators of $R$. Thus, we consider the $x$-deformation given by $g_ie_i=xe_i$ for all $i$. Multiplying \eqref{bt5} by $e_i$, we obtain $x^2e_i=e_i+(u-1)e_i+(v-1)xe_i$. Therefore, the parameter $x$ must satisfy the quadratic equation $x^2-(v-1)x-u=0$. In particular, one solution is $x=u=v$, and so we make the following definition.

\begin{definition}
Set $\E_n^P(\sqrt{p}):=\E_n(\sqrt{p})/I$, where $I$ is the two-sided ideal generated by $g_ie_i-\sqrt{p}e_i$ for all $i$. Whenever no confusion arises, we will not distinguish between $g_i$ (resp. $e_i$) and their images in $\E_n^P(\sqrt{p})$.
\end{definition}

\begin{remark}
Note that $e_2g_2=(g_1g_2e_1g_2^{-1}g_1^{-1})g_2=g_1g_2e_1g_1g_2^{-1}g_1^{-1}=g_1g_2e_1g_2^{-1}g_1^{-1}=e_2$. Hence, an inductive argument shows that $I$ is generated by $g_1e_1-\sqrt{p}e_1$.
\end{remark}

\begin{proposition}\label{P-H/bt}
The map $g_i\mapsto H_i$, $e_i\mapsto q^{-2}F_i$ defines an isomorphism from $\E_n^P(\sqrt{p})$ to $\Pa_n'(p,q)$.
\end{proposition}
\begin{proof}
To show that the map defines a homomorphism of algebras, it suffices to check that the $H_i$'s and $F_i$'s satisfy the defining relations of $\E_n(u)$ together with $e_ig_i=ue_i$ for all $i$. These verifications follow directly, except for relations in \eqref{bt4} and \eqref{bt5}. Since $u=v=\sqrt{p}$, relation \eqref{bt5} becomes $g_i^2=1+(\sqrt{p}-1)e_i(1+g_i)$. Under our map, the left-hand side goes to $H_i^2=1+q^{-2}(p-1)F_i$, while the right-hand side goes to 
\[1+(\sqrt{p}-1)q^{-2}F_i(1+H_i) \stackrel{\eqref{H^2}}{=}1+q^{-2}(p-1)F_i.\]
 So, the relation \eqref{bt5} is preserved.

For the first relation in \eqref{bt4} we need to prove $H_iH_jF_i=F_jH_iH_j$ for $|i-j|>1$, or equivalently, $H_jF_iH_j^{-1}=H_i^{-1}F_jH_i$. Indeed,\begin{eqnarray*}
H_jF_iH_j^{-1}&\stackrel{\eqref{H^-1}}{=}& H_jF_i(H_j-q^{-2}(\sqrt{p}-\sqrt{p}^{\,-1})F_j)\\
&\stackrel{\eqref{BP2}}{=}&H_iF_jH_i-q^{-2}(\sqrt{p}-\sqrt{p}^{\,-1})H_jF_jF_i\\
&\stackrel{\eqref{H^2}}{=}&
H_iF_jH_i-q^{-2}(\sqrt{p}-\sqrt{p}^{\,-1})\sqrt{p}F_jF_i.
\end{eqnarray*}
Since $\sqrt{p}F_jF_i=F_j\sqrt{p}F_i \stackrel{\eqref{H^2}}{=}F_jF_i H_i=F_iF_jH_i$, then
\begin{eqnarray*}
H_jF_iH_j^{-1}&\stackrel{\eqref{H^-1}}{=} &H_iF_jH_i-q^{-2}(\sqrt{p}-\sqrt{p}^{\,-1}) F_iF_jH_i\\
&=&(H_i-q^{-2}(\sqrt{p}-\sqrt{p}^{\,-1})F_i)F_jH_i\,\,\,\stackrel{\eqref{H^-1}}{=}\,\,\,H_i^{-1}F_jH_i.
\end{eqnarray*}
So, the first relation in \eqref{bt4} is preserved.

For the second relation in \eqref{bt4}, we begin by pointing out that to prove this relation, it suffices to first show $H_iF_iF_j=F_jH_i F_j$ when $|i-j|=1$, or equivalently, $F_iF_j=H_i^{-1}F_jH_iF_j$. Indeed,
\begin{eqnarray*}
H_i^{-1}F_jH_iF_j&\stackrel{\eqref{H^-1}}{=}&(H_i-q^{-2}(\sqrt{p}-\sqrt{p}^{\,-1})F_i)F_jH_iF_j\\
&=&H_iF_jH_iF_j-q^{-2}(\sqrt{p}-\sqrt{p}^{\,-1})F_iF_jH_iF_j\\
&\stackrel{\eqref{BP2},\eqref{Par1}}{=}&
H_jF_iH_jF_j-q^{-2}(\sqrt{p}-\sqrt{p}^{\,-1})F_jF_iH_iF_j\\
&\stackrel{\eqref{H^2}}{=}&\sqrt{p}H_jF_iF_j-q^{-2}(\sqrt{p}-\sqrt{p}^{\,-1})\sqrt{p}F_jF_iF_j.
\end{eqnarray*}
Since $\sqrt{p}H_jF_iF_j=\sqrt{p}H_jF_jF_i\stackrel{\eqref{H^2}}{=}pF_iF_j$ and $F_jF_iF_j=q^{2}F_iF_j$, we obtain\[H_i^{-1}F_jH_iF_j=pF_iF_j-(\sqrt{p}-\sqrt{p}^{\,-1})\sqrt{p}F_iF_j=F_iF_j.\]Similarly, one proves $F_iF_jH_i=F_jH_iF_j$. So, the second relation in \eqref{bt4} is preserved.

In summary, the map is a homomorphism of algebras. Moreover, it is bijective, with inverse given by $H_i\mapsto g_i$ and $F_i\mapsto q^2e_i$. Hence, the proof is concluded.
\end{proof}

\subsection{Two Temperley--Lieb-like quotients}\label{TwoTL}

In this subsection, we use the presentation $\Pa_n'(p,q)$ of the Party-Hecke algebra. 
Define
\begin{equation}\label{defTi}T_i=\frac{1}{2}\Big(H_i+q^{-2}(1-\sqrt{p})F_i+1\Big)\in
\Pa_n'(p, q).
\end{equation}
A straightforward computation shows that $T_i^2=T_i$. We now have the following result.
\begin{proposition}\label{PreIdem}
The idempotents $T_1,\ldots,T_{n-1}$, together with $F_1,\ldots,F_{n-1}$ define a presentation $\Pa(q)$ of the Party-Hecke algebra, whose defining relations are:
\begin{align}
T_i^2=T_i\,,&\quad T_iT_j=T_jT_i\quad\text{for }|i-j|>1,\label{Idemp1}\\
4T_iT_jT_i- T_i &= 4T_jT_iT_j-T_j\quad\text{for }|i-j|=1,\label{Idemp2}\\
F_i^2=q^2F_i\,,&\quad F_iF_j=F_jF_i\,,\label{Idemp3}\\ 
T_iF_i=F_iT_i=F_i,&\quad T_iF_j=F_jT_i\quad\text{for }|i-j|>1,\label{Idemp4}\\
4T_iT_jF_i-2T_j F_i-F_i & = 4F_jT_iT_j-2F_j T_i -F_j\quad\text{for }|i-j|=1.\label{Idemp5}
\end{align} 
\end{proposition}
\begin{proof}
Since each $T_i$ is a linear combination of $1$, $H_i$ and $F_i$, it follows that the Party-Hecke algebra is linearly generated by the $T_i$'s and $F_i$'s. Moreover, it is routine to check that the defining relations of $\Pa_n'(p,q)$ correspond, in terms of $T_i$ and $F_i$, precisely to relations \eqref{Idemp1}--\eqref{Idemp5}. This completes the proof.
\end{proof} 
The presentation in Proposition~\ref{PreIdem} is analogous to one of the Iwahori--Hecke algebra that yields the Temperley--Lieb algebra as a quotient \cite{TL1971,Jones83}. It is therefore natural to define a Temperley--Lieb quotient of the Party-Hecke algebra via the presentation of Proposition~\ref{PreIdem}. In this way, two quotients arise, namely\[\Pa_n(q)/I\qquad \text{and}\qquad\Pa_n(q)/J,\]where $I$ (resp. $J$) denotes the two-sided ideal generated by 
\begin{equation}\label{IdealGens}4T_iT_jF_i-2T_jF_i-F_i\quad\text{(resp. }4T_iT_jT_i-T_i\text{)}\quad\text{if}\quad|i-j|=1.\quad\end{equation}

\begin{proposition}\label{HP/I}
The quotient $\Pa_n(q)/I$ is the algebra $\H_n(u)$ with $u=7\pm4\sqrt{3}$.
\end{proposition}
\begin{proof}
Multiplying the equation $4T_iT_jF_i=2T_jF_i+F_i$ on the left by $T_i$, and using \eqref{Idemp1} and \eqref{Idemp4}, we obtain $4T_iT_jF_i=2T_iT_jF_i+F_i$. Hence, $2T_iT_jF_i=F_i$. Substituting this result into the original equation $4T_iT_jF_i=2T_jF_i+F_i$, we find that $2T_jF_i=F_i$. Multiplying this equation on the left by $T_j$ gives $T_jF_i=0$, and hence $F_i=0$. By \eqref{heckeRel3} and \eqref{heckeRel4}, this shows that $\H_n(u)$ is obtained with $16=(u+1)^2u^{-1}$, that is, $u=7\pm4\sqrt{3}$.
\end{proof}

\begin{proposition}\label{HP/J}
The quotient $\Pa_n(q)/J$ is the Temperley--Lieb algebra $TL_n(4)$, that is, the algebra presented by generators $T_1,\ldots,T_{n-1}$ subject to \eqref{Idemp1} together with the relations:\begin{equation} 4T_iT_jT_i=T_i\quad\text{for }|i-j|=1.\label{TempL}\end{equation}
\end{proposition}
\begin{proof}
Multiplying the equation $4T_iT_jT_i=T_i$ (resp. $4T_jT_iT_j=T_j$) on the right by $F_i$ (resp. $F_j$) and using \eqref{Idemp4}, we get $4T_iT_jF_i=F_i$ (resp. $4T_jT_iF_j=F_j$). Using these two relations in \eqref{Idemp5}, we obtain\begin{equation}\label{EqHP/I}T_jF_i=F_jT_i,\quad\text{for }|i-j|=1.\end{equation}On the other hand, multiplying again $4T_iT_jT_i=T_i$ (resp. $4T_jT_iT_j=T_j$) on the right (resp. on the left) by $F_i$ (resp. $F_j$), we obtain $4T_iT_jF_i=F_i$ (resp. $4F_jT_iT_j = F_j$). Now, using 
\eqref{EqHP/I}, we deduce $4T_iF_jT_i=F_i$ (resp. $4T_jF_iT_j=F_j$). Multiplying these last equations on the right and left by $F_i$ (resp. $F_j$), we get $4F_iF_j=F_i$ (resp. $4F_iF_j=F_j$). Hence $16F_iF_j=F_iF_j$ by \eqref{Idemp3}. Thus $F_iF_j=0$, and therefore $F_i=0$ for all $i$. Consequently, the quotient is presented by the $T_i$'s with the defining relations \eqref{Idemp1} and \eqref{TempL}. Hence, by \eqref{heckeRel3} and \eqref{heckeRel4}, the proof follows. 
\end{proof}

\subsection{A partition Temperley--Lieb-like quotient}\label{QuoPTL}

Motivated by the definition of the partition Temperley--Lieb algebra \cite{Juyumaya13} (cf.~\cite{Ry-HaJPAA22}) and Proposition~\ref{P-H/bt}, it is natural to consider, for $n\geq3$, the quotient $\Pa_n(p,q)/\mathrm{F}$, where $\mathrm{F}$ is the two-sided ideal generated by the elements $F_iF_j$ with $|i-j|=1$. By Lemma~\ref{BTRels}(2), $\mathrm{F}$ is principal and generated by any $F_iF_{i+1}$; thus we may set $\mathrm{F}=\langle F_1F_2\rangle$.

Recall that $H_i=(q\sqrt{p})^{-1}G_i$. The elements $T_i$ defined in \eqref{defTi} can be written as a linear combination of $1$, $G_i$, and $F_i$. Therefore, the mapping 
$G_i\mapsto T_i$ and $F_i\mapsto F_i$ defines an automorphism of the Party-Hecke algebra that preserves the ideal $\mathrm{F}$. We define\[\PPa_n(q):=\Pa_n(q)/\mathrm{F}=\Pa_n(p,q)/\mathrm{F}.\]Using the \href{https://magma.maths.usyd.edu.au/calc/}{MAGMA} computer algebra system, we find that the dimensions of $\PPa_n(q)$ are $15$, $114$, $1170$, and $15570$ for $n=3,4,5$, and $6$, respectively. This suggests that the dimensions of $\PPa_n(q)$ correspond to the sequence \href{https://oeis.org/A346224}{A346224}.

\section{Future work}\label{Future}

The structural results established in this paper lay the groundwork for further research directions. In this section, we briefly outline two such directions: the representation theory of the related monoid $T\Sfr_n$, and the potential use of the Party-Hecke algebra to define a Jones-type invariant for virtual knots.

\subsection{Irreducible representations of $T\Sfr_n$}\label{RepTSn}

Just as the maximal subgroups of $T\Sfr_n$ are determined by the maximal subgroups of $\P_n$, it is natural to expect that the irreducible representations of $T\Sfr_n$ can be obtained by applying the methods in \cite[Section 3]{OrSaSchZaAl2022}. As a first step, we note that, similarly to the party monoid \cite[Proposition 3.5]{OrSaSchZaAl2022}, the $J$-classes of $T\Sfr_n$ are also indexed by the partitions of $n$. Indeed, for $e,e'\in P_n$, one has $e'=s(e)=ses^{-1}$ for some $s\in\Sfr_n$ if and only if $T\Sfr_ne'T\Sfr_n=T\Sfr_nses^{-1}T\Sfr_n=T\Sfr_neT\Sfr_n$, since $s$ is invertible. Thus, by Remark~\ref{orbitPartition},\[e\equiv_Je'\quad\text{if and only if}\quad\|e\|=\|e'\|.\]On the other hand, if $g=es\in T\Sfr_n$ with $e\in P_n$ and $s\in\Sfr_n$, then $g\equiv_Je$ since $T\Sfr_ngT\Sfr_n=T\Sfr_nesT\Sfr_n=T\Sfr_neT\Sfr_n$. Consequently, if $g=es$ and $h=e's'$ with $\|e\|=\|e'\|$, then $g\equiv_Jh$; conversely, if $g\equiv_Jh$, then $\|e\|=\|e'\|$ since $e\equiv_Jg\equiv_Jh\equiv_Je'$. Hence,\[g\equiv_Jh\quad\text{if and only if}\quad\|e\|=\|e'\|.\]This establishes the claimed bijection.

Following this observation, for each partition $\lambda$ of $n$, we write $J_\lambda$ for the $J$-class of any $es\in T\Sfr_n$ with $\|e\|=\lambda$. Consequently,\[T\Sfr_n=\bigsqcup_{\lambda\in\Par_n}J_\lambda,\]where $\Par_n$ denotes the set of partitions of $n$.

Furthermore, for $e\in P_n$ with $\|e\|=\lambda$, one can easily check that $J_\lambda=\{ses'\mid s,s'\in\Sfr_n\}$. To see this, note that for $s,s'\in\Sfr_n$, we have $T\Sfr_nses'T\Sfr_n=T\Sfr_neT\Sfr_n$, so $ses'\equiv_Je$ in $J_\lambda$. Conversely, if $e's'\in J_\lambda$ with $e'\in P_n$ and $s'\in\Sfr_n$, then there exists $s\in\Sfr_n$ such that $e'=ses^{-1}$, whence $e's'=se(s^{-1}s')$. See Figure~\ref{JClassesTS3} for an illustration of this structure.
\begin{figure}[H]\figtwe\caption{$J$-classes of $T\Sfr_3$.}\label{JClassesTS3}\end{figure}

\subsection{Jones-type invariants of virtual knots}\label{Virtual}

Because the Party-Hecke algebra is a quotient of the bt-algebra (Proposition~\ref{TBnQuotient}), which is indeed a knot algebra, one might wonder if the Party-Hecke algebra is also a knot algebra. In fact, we believe that the Party-Hecke algebra could be a knot algebra, as it could define a Jones-type invariant for virtual knots. This is based on the fact that the trace supported by the bt-algebra can be shown to pass to the Party-Hecke algebra (cf. \cite[Theorem 7.7]{Juyumaya13}, \cite[Section 5]{GoJuKoLa}). This property, together with the representation of the virtual braid group explained below, may enable the construction of such an invariant.

For $a,b\in \CC$, define $V_i = a H_i + b F_i$. A direct computation shows that the following relations hold in $\Pa_n'(p,q)$. 
\begin{enumerate}
\item $V_i^2=q^{-2}(pa^2+2\sqrt{p}q^2ab+q^4b^2-a^2)F_i+a^2$.
\item $V_iV_jV_i=V_jV_iV_j$, for $|i-j|=1$.
\item $V_iV_jH_i=H_jV_iV_j$, for $|i-j|=1$.
\end{enumerate}
Recall that the virtual braid group $V\!B_n$ is defined by the generators $\sigma_1,\ldots,\sigma_{n-1}$, which satisfy the braid relations \eqref{braidRels}, the generators $s_1,\ldots, s_{n-1}$, which satisfy the symmetric group relations \eqref{TSn1}, along with the following mixed relations:
\begin{equation}\label{VB1}
\sigma_is_j=s_j\sigma_i\quad\text{if }|i-j|> 1,\quad\text{and}\quad s_is_j\sigma_i=\sigma_js_is_j\quad \text{if }|i-j|= 1. 
\end{equation}
One proves 
the mapping $\sigma_i\mapsto H_i$, $s_i\mapsto V_i$ defines a representation of $V\!B_n$ in $\Pa_n'(p,q)$, if\[a=1\quad\text{and}\quad b=\frac{1\pm\sqrt{p}}{q^2},\quad\text{or}\quad a=-1\quad \text{and}\quad b=-\frac{1\pm\sqrt{p}}{q^2}.\]
}

\end{document}

%% file: pics/pics.tex
\usepackage[
	externalize=0,
	font=0.5,
	height=0.9,
    strwidth=0.5,
	width=0.4,
]{pics/strands20231225}

\newcommand{\aspdf}{1}
\newcommand{\vcdraw}[1]{\vcenter{\hbox{#1}}}
\newcommand{\type}{2} 
\usetikzlibrary{cd}

\newcommand{\figttn}{
	\centering
	\ifnum\aspdf=1$
        \vcdraw{\includegraphics{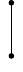}}\!\cdots\!
        \vcdraw{\includegraphics{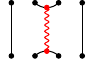}}\,\cdots\!
        \vcdraw{\includegraphics{pics/line.pdf}}
	$\else
		\begin{tikzpicture}        
        	\vpartition[type=0,tkzpic=0]{{1,-1},{2,3},{-2,-3},{4,-4}}
        	\tie[style=solid,snake=true,snakelen=3]{{2.5,0.09},{2.5,0.91}}
        \end{tikzpicture}
	\fi
}

\newcommand{\figtwe}{
	\centering
	\ifnum\aspdf=1\begin{gather*}
		J_{(1,1,1)}=\left\{\begin{array}{llllll}
        	\vcdraw{\includegraphics[scale=0.8]{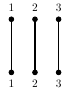}}\,,&
        	\vcdraw{\includegraphics[scale=0.8]{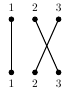}}\,,&
        	\vcdraw{\includegraphics[scale=0.8]{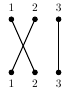}}\,,&
        	\vcdraw{\includegraphics[scale=0.8]{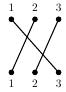}}\,,&
        	\vcdraw{\includegraphics[scale=0.8]{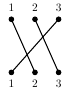}}\,,&
        	\vcdraw{\includegraphics[scale=0.8]{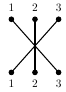}}
		\end{array}\right\},\\[0.2cm]
		J_{(2,1)}=\left\{\begin{array}{lllllllll}
			\vcdraw{\includegraphics[scale=0.8]{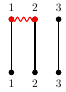}}\,,&
			\vcdraw{\includegraphics[scale=0.8]{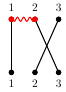}}\,,&
			\vcdraw{\includegraphics[scale=0.8]{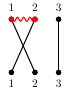}}\,,&
			\vcdraw{\includegraphics[scale=0.8]{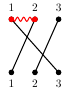}}\,,&
			\vcdraw{\includegraphics[scale=0.8]{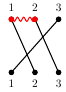}}\,,&
			\vcdraw{\includegraphics[scale=0.8]{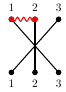}}\,,&
			\vcdraw{\includegraphics[scale=0.8]{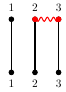}}\,,&
			\vcdraw{\includegraphics[scale=0.8]{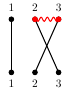}}\,,&
			\vcdraw{\includegraphics[scale=0.8]{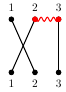}}\,,\\[0.6cm]
			\vcdraw{\includegraphics[scale=0.8]{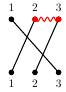}}\,,&
			\vcdraw{\includegraphics[scale=0.8]{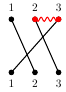}}\,,&
			\vcdraw{\includegraphics[scale=0.8]{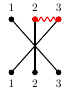}}\,,&
			\vcdraw{\includegraphics[scale=0.8]{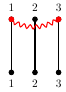}}\,,&
			\vcdraw{\includegraphics[scale=0.8]{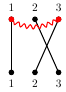}}\,,&
			\vcdraw{\includegraphics[scale=0.8]{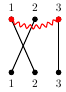}}\,,&
			\vcdraw{\includegraphics[scale=0.8]{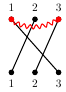}}\,,&
			\vcdraw{\includegraphics[scale=0.8]{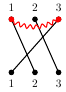}}\,,&
			\vcdraw{\includegraphics[scale=0.8]{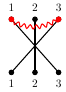}}
		\end{array}\right\},\\[0.2cm]
		J_{(3)}=\left\{\begin{array}{llllll}
        	\vcdraw{\includegraphics[scale=0.8]{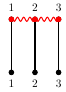}}\,,&
        	\vcdraw{\includegraphics[scale=0.8]{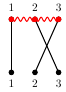}}\,,&
        	\vcdraw{\includegraphics[scale=0.8]{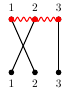}}\,,&
        	\vcdraw{\includegraphics[scale=0.8]{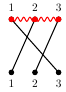}}\,,&
        	\vcdraw{\includegraphics[scale=0.8]{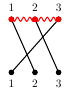}}\,,&
        	\vcdraw{\includegraphics[scale=0.8]{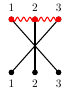}}
		\end{array}\right\}.
	\end{gather*}\else
		\begin{tikzpicture}
			\permutation[tkzpic=0]{1,2,3}
		\end{tikzpicture}
		\begin{tikzpicture}
			\permutation[tkzpic=0]{1,3,2}
		\end{tikzpicture}
		\begin{tikzpicture}
			\permutation[tkzpic=0]{2,1,3}
		\end{tikzpicture}
		\begin{tikzpicture}
			\permutation[tkzpic=0]{3,1,2}
		\end{tikzpicture}
		\begin{tikzpicture}
			\permutation[tkzpic=0]{2,3,1}
		\end{tikzpicture}
		\begin{tikzpicture}
			\permutation[tkzpic=0]{3,2,1}
		\end{tikzpicture}
		\begin{tikzpicture}
			\permutation[tkzpic=0]{1,2,3}
			\tie[style=solid,snake=true,snakelen=3]{{1,1},{2,1},{3,1}}
		\end{tikzpicture}
		\begin{tikzpicture}
			\permutation[tkzpic=0]{1,3,2}
			\tie[style=solid,snake=true,snakelen=3]{{1,1},{2,1},{3,1}}
		\end{tikzpicture}
		\begin{tikzpicture}
			\permutation[tkzpic=0]{2,1,3}
			\tie[style=solid,snake=true,snakelen=3]{{1,1},{2,1},{3,1}}
		\end{tikzpicture}
		\begin{tikzpicture}
			\permutation[tkzpic=0]{3,1,2}
			\tie[style=solid,snake=true,snakelen=3]{{1,1},{2,1},{3,1}}
		\end{tikzpicture}
		\begin{tikzpicture}
			\permutation[tkzpic=0]{2,3,1}
			\tie[style=solid,snake=true,snakelen=3]{{1,1},{2,1},{3,1}}
		\end{tikzpicture}
		\begin{tikzpicture}
			\permutation[tkzpic=0]{3,2,1}
			\tie[style=solid,snake=true,snakelen=3]{{1,1},{2,1},{3,1}}
		\end{tikzpicture}
		\begin{tikzpicture}
			\permutation[tkzpic=0]{1,2,3}
			\tie[style=solid,snake=true,snakelen=3]{{1,1},{2,1}}
		\end{tikzpicture}
		\begin{tikzpicture}
			\permutation[tkzpic=0]{1,3,2}
			\tie[style=solid,snake=true,snakelen=3]{{1,1},{2,1}}
		\end{tikzpicture}
		\begin{tikzpicture}
			\permutation[tkzpic=0]{2,1,3}
			\tie[style=solid,snake=true,snakelen=3]{{1,1},{2,1}}
		\end{tikzpicture}
		\begin{tikzpicture}
			\permutation[tkzpic=0]{3,1,2}
			\tie[style=solid,snake=true,snakelen=3]{{1,1},{2,1}}
		\end{tikzpicture}
		\begin{tikzpicture}
			\permutation[tkzpic=0]{2,3,1}
			\tie[style=solid,snake=true,snakelen=3]{{1,1},{2,1}}
		\end{tikzpicture}
		\begin{tikzpicture}
			\permutation[tkzpic=0]{3,2,1}
			\tie[style=solid,snake=true,snakelen=3]{{1,1},{2,1}}
		\end{tikzpicture}
		\begin{tikzpicture}
			\permutation[tkzpic=0]{1,2,3}
			\tie[style=solid,snake=true,snakelen=3]{{2,1},{3,1}}
		\end{tikzpicture}
		\begin{tikzpicture}
			\permutation[tkzpic=0]{1,3,2}
			\tie[style=solid,snake=true,snakelen=3]{{2,1},{3,1}}
		\end{tikzpicture}
		\begin{tikzpicture}
			\permutation[tkzpic=0]{2,1,3}
			\tie[style=solid,snake=true,snakelen=3]{{2,1},{3,1}}
		\end{tikzpicture}
		\begin{tikzpicture}
			\permutation[tkzpic=0]{3,1,2}
			\tie[style=solid,snake=true,snakelen=3]{{2,1},{3,1}}
		\end{tikzpicture}
		\begin{tikzpicture}
			\permutation[tkzpic=0]{2,3,1}
			\tie[style=solid,snake=true,snakelen=3]{{2,1},{3,1}}
		\end{tikzpicture}
		\begin{tikzpicture}
			\permutation[tkzpic=0]{3,2,1}
			\tie[style=solid,snake=true,snakelen=3]{{2,1},{3,1}}
		\end{tikzpicture}
		\begin{tikzpicture}
			\permutation[tkzpic=0]{1,2,3}
			\tie[style=solid,snake=true,snakelen=3,bend=-30]{{1,1},{3,1}}
		\end{tikzpicture}
		\begin{tikzpicture}
			\permutation[tkzpic=0]{1,3,2}
			\tie[style=solid,snake=true,snakelen=3,bend=-30]{{1,1},{3,1}}
		\end{tikzpicture}
		\begin{tikzpicture}
			\permutation[tkzpic=0]{2,1,3}
			\tie[style=solid,snake=true,snakelen=3,bend=-30]{{1,1},{3,1}}
		\end{tikzpicture}
		\begin{tikzpicture}
			\permutation[tkzpic=0]{3,1,2}
			\tie[style=solid,snake=true,snakelen=3,bend=-30]{{1,1},{3,1}}
		\end{tikzpicture}
		\begin{tikzpicture}
			\permutation[tkzpic=0]{2,3,1}
			\tie[style=solid,snake=true,snakelen=3,bend=-30]{{1,1},{3,1}}
		\end{tikzpicture}
		\begin{tikzpicture}
			\permutation[tkzpic=0]{3,2,1}
			\tie[style=solid,snake=true,snakelen=3,bend=-30]{{1,1},{3,1}}
		\end{tikzpicture}
	\fi
}

\newcommand{\figele}{
	\centering
	\ifnum\aspdf=1$
		\left(\,\,\,
        \vcdraw{\includegraphics{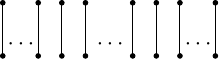}}
        \,\,,\,\,
        \vcdraw{\includegraphics{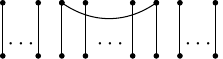}}
        \,\,\right)\,=\,
        \vcdraw{\includegraphics{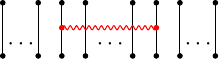}}
	$\else
		\begin{tikzpicture}
			\tie[style=solid,color=black]{{-0.5,0},{-0.5,1}}
			\tie[style=solid,color=black]{{1,0},{1,1}}
			\tie[style=solid,color=black]{{2,0},{2,1}}
			\tie[style=solid,color=black]{{3,0},{3,1}}
			\tie[style=solid,color=black]{{5,0},{5,1}}
			\tie[style=solid,color=black]{{6,0},{6,1}}
			\tie[style=solid,color=black]{{7,0},{7,1}}
			\tie[style=solid,color=black]{{8.5,0},{8.5,1}}
			\node at(-0.26,0.2){$\cdots$};
			\node at(1.25,0.2){$\cdots$};
			\node at(2.74,0.2){$\cdots$};
		\end{tikzpicture}
		\begin{tikzpicture}
			\tie[style=solid,color=black]{{-0.5,0},{-0.5,1}}
			\tie[style=solid,color=black]{{1,0},{1,1}}
			\tie[style=solid,color=black]{{2,0},{2,1}}
			\tie[style=solid,color=black]{{3,0},{3,1}}
			\tie[style=solid,color=black]{{5,0},{5,1}}
			\tie[style=solid,color=black]{{6,0},{6,1}}
			\tie[style=solid,color=black]{{7,0},{7,1}}
			\tie[style=solid,color=black]{{8.5,0},{8.5,1}}
			\tie[color=black,style=solid,bend=-35]{{2,1},{6,1}}
			\node at(-0.26,0.2){$\cdots$};
			\node at(1.25,0.2){$\cdots$};
			\node at(2.74,0.2){$\cdots$};
		\end{tikzpicture}
		\begin{tikzpicture}
			\tie[style=solid,color=black]{{-0.5,0},{-0.5,1}}
			\tie[style=solid,color=black]{{1,0},{1,1}}
			\tie[style=solid,color=black]{{2,0},{2,1}}
			\tie[style=solid,color=black]{{3,0},{3,1}}
			\tie[style=solid,color=black]{{5,0},{5,1}}
			\tie[style=solid,color=black]{{6,0},{6,1}}
			\tie[style=solid,color=black]{{7,0},{7,1}}
			\tie[style=solid,color=black]{{8.5,0},{8.5,1}}
			\tie[snake=true,snakelen=3,style=solid]{{2,0.53},{6,0.53}}
			\node at(-0.26,0.2){$\cdots$};
			\node at(1.25,0.2){$\cdots$};
			\node at(2.74,0.2){$\cdots$};
		\end{tikzpicture}
	\fi
}

\newcommand{\figten}{
	\centering
	\ifnum\aspdf=1$
        \vcdraw{\includegraphics{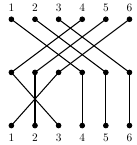}}
        \,\,\,=\,\,
        \vcdraw{\includegraphics{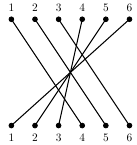}}
        \,\,\,=\,\,
        \vcdraw{\includegraphics{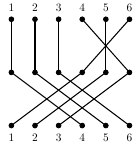}}
	$\else
		\begin{tikzpicture}
			\vpartition[type=1,tkzpic=0,floor=1]
			{{1,-4},{2,-5},{3,-6},{4,-1},{5,-2},{6,-3}}
			\vpartition[type=-1,tkzpic=0,bulla=0]
			{{1,-3},{2,-2},{3,-1},{4,-4},{5,-5},{6,-6}}
		\end{tikzpicture}
		\begin{tikzpicture}
			\vpartition[type=2,tkzpic=0,height=1.8]
			{{1,-4},{2,-5},{3,-6},{4,-3},{5,-2},{6,-1}}
		\end{tikzpicture}
		\begin{tikzpicture}
			\vpartition[type=1,tkzpic=0,floor=1]
			{{1,-1},{2,-2},{3,-3},{4,-6},{5,-5},{6,-4}}
			\vpartition[type=-1,tkzpic=0,bulla=0]
			{{1,-4},{2,-5},{3,-6},{4,-1},{5,-2},{6,-3}}
		\end{tikzpicture}
	\fi
}

\newcommand{\fignin}{
	\centering
	\ifnum\aspdf=1$\begin{array}{ccc}
        \quad\includegraphics{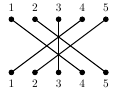}\quad&
        \quad\includegraphics{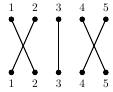}\quad&
        \quad\includegraphics{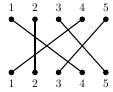}\quad\\
        s_{12|45}&s_{15|24}&s_{13|45}
	\end{array}$\else
		\vpartition[type=2]{{3,-3},{1,-4},{2,-5},{4,-1},{5,-2}}
		\vpartition[type=2]{{3,-3},{1,-2},{5,-4},{2,-1},{4,-5}}
		\vpartition[type=2]{{2,-2},{1,-4},{3,-5},{4,-1},{5,-3}}
	\fi
}

\newcommand{\figeig}{
	\centering
	\ifnum\aspdf=1$
            g\,\,=\vcdraw{\includegraphics{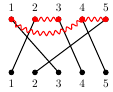}}\qquad\,\,
            g^*\,=\vcdraw{\includegraphics{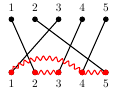}}
	$\else
		\begin{tikzpicture}
        	\vpartition[type=2,tkzpic=0]{{1,-3},{4,-5},{5,-2},{2,-1},{3,-4}}
			\tie[snake=true,snakelen=3,style=solid,bend=-40]{{1,1},{4,1}}
			\tie[snake=true,snakelen=3,style=solid]{{4,1},{5,1}}
			\tie[snake=true,snakelen=3,style=solid]{{2,1},{3,1}}
        \end{tikzpicture}
		\begin{tikzpicture}
        	\vpartition[type=2,tkzpic=0]{{-1,3},{-4,5},{-5,2},{-2,1},{-3,4}}
			\tie[snake=true,snakelen=3,style=solid,bend=40]{{1,0},{4,0}}
			\tie[snake=true,snakelen=3,style=solid]{{4,0},{5,0}}
			\tie[snake=true,snakelen=3,style=solid]{{2,0},{3,0}}
        \end{tikzpicture}
	\fi
}

\newcommand{\figsev}{
	\centering
	\ifnum\aspdf=1$
            g\,\,=\vcdraw{\includegraphics{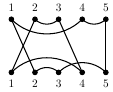}}\qquad\,\,
            g^*\,=\vcdraw{\includegraphics{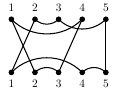}}
	$\else
        \vpartition[type=2]{{1,4,5,-5,-3,-2,1},{2,3,-4,-1,2}}
        \vpartition[type=2]{{-1,-4,-5,5,3,2,-1},{-2,-3,4,1,-2}}
	\fi
}

\newcommand{\figsix}{
	\centering
	\ifnum\aspdf=1$
        \vcdraw{\includegraphics{pics/line.pdf}}\!\cdots\!
        \vcdraw{\includegraphics{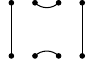}}\,\cdots\!
        \vcdraw{\includegraphics{pics/line.pdf}}
	$\else
        \vpartition[type=0]{{1,-1}}
        \vpartition[type=0]{{1,-1},{2,3},{-2,-3},{4,-4}}
	\fi
}

\newcommand{\figfiv}{
	\centering
	\ifnum\aspdf=1$
        \begin{array}{c}
            \includegraphics{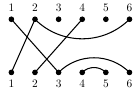}\\
            _{(\{1,9,12\},\{2,6,7\},\{3\},\{4,8\},\{5\},\{10,11\})}
        \end{array}
	$\else
        \vpartition[type=2]{{-1,2,6},{1,-3,-6},{4,-2},{-4,-5}}
	\fi
}

\newcommand{\figfou}{
	\centering
	\ifnum\aspdf=1$
        \vcdraw{\includegraphics{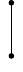}}\kern-0.25em\cdots\kern-0.25em
        \vcdraw{\includegraphics{pics/021.pdf}}\kern+0.12em
        \vcdraw{\includegraphics{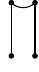}}\kern+0.38em
        \vcdraw{\includegraphics{pics/021.pdf}}\kern-0.25em\cdots\kern-0.25em
        \vcdraw{\includegraphics{pics/021.pdf}}
        \,\,=\,\,
        \vcdraw{\includegraphics{pics/021.pdf}}\kern-0.25em\cdots\kern-0.25em
        \vcdraw{\includegraphics{pics/021.pdf}}\kern+0.12em
        \vcdraw{\includegraphics{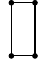}}\kern+0.38em
        \vcdraw{\includegraphics{pics/021.pdf}}\kern-0.25em\cdots\kern-0.25em
        \vcdraw{\includegraphics{pics/021.pdf}}
	$\else
        \vpartition[type=0]{{1,-1}}
        \vpartition[type=0]{{-1,1,2,-2}}
        \vpartition[type=0]{{-1,1,2,-2,-1}}
	\fi
}

\newcommand{\figthr}{
	\centering
	\ifnum\aspdf=1$
        \begin{array}{ccccccc}
            \vcdraw{\includegraphics{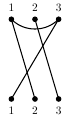}}&=&
            \vcdraw{\includegraphics{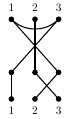}}&\stackrel{\eqref{PartyDual2}}{=}&
            \vcdraw{\includegraphics{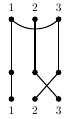}}&=&
            \vcdraw{\includegraphics{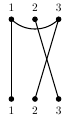}}\\[-0.1cm]
            \,\,\,_{f_{1,3}s_2s_1}&&\,\,_{f_{1,3}s_{1,3}s_2}&&\,\,_{f_{1,3}s_2}&&\,\,\,_{f_{1,3}s_2}
        \end{array}
	$\else$
        \begin{array}{ccccccc}
            \vcdraw{\vpartition[type=\type,height=1.35]{{-1,3,1,-2},{2,-3}}}&
            \to&
            \vcdraw{\begin{tikzpicture}
            \vpartition[type=1,tkzpic=0,floor=0.5]{{-1,3,1,-3},{2,-2}}
            \vpartition[type=-1,tkzpic=0,bulla=0,height=0.45]{{1,-1},{2,-3},{3,-2}}
            \end{tikzpicture}}&
            \to&
            \vcdraw{\begin{tikzpicture}
            \vpartition[type=1,tkzpic=0,floor=0.5]{{-1,1,3,-3},{2,-2}}
            \vpartition[type=-1,tkzpic=0,bulla=0,height=0.45]{{1,-1},{2,-3},{3,-2}}
            \end{tikzpicture}}&
            \to&
            \vcdraw{\vpartition[type=\type,height=1.35]{{-1,1,3,-2},{2,-3}}}\\[-0.1cm]
            \,\,\,_{f_{1,3}s_2s_1}&&\,\,_{f_{1,3}s_{1,3}s_2}&&\,\,_{f_{1,3}s_2}&&\,\,\,_{f_{1,3}s_2}
        \end{array}
	$\fi
}

\newcommand{\figtwo}{
	\centering
	\ifnum\aspdf=1$
        \begin{array}{cccccc}
            \includegraphics{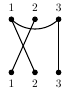}&
            \includegraphics{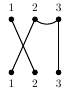}&
            \includegraphics{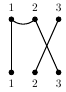}&
            \includegraphics{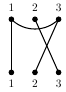}&
            \includegraphics{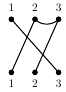}&
            \includegraphics{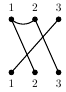}\\[-0.1cm]
            \,\,_{f_{1,3}s_1}&\,\,_{f_{2,3}s_1}&\,\,_{f_{1,2}s_2}&\,\,_{f_{1,3}s_2}&\,\,_{f_{2,3}s_1s_2}&\,\,_{f_{1,2}s_2s_1}
        \end{array}
	$\else$
        \begin{array}{cccccc}
            \vpartition[type=\type]{{-2,1,3,-3},{2,-1}}&
            \vpartition[type=\type]{{-1,2,3,-3},{1,-2}}&
            \vpartition[type=\type]{{-1,1,2,-3},{3,-2}}&
            \vpartition[type=\type]{{-1,1,3,-2},{2,-3}}&
            \vpartition[type=\type]{{1,-3},{-1,2,3,-2}}&
            \vpartition[type=\type]{{-2,1,2,-3},{3,-1}}\\[-0.1cm]
            \,\,_{f_{1,3}s_1}&\,\,_{f_{2,3}s_1}&\,\,_{f_{1,2}s_2}&\,\,_{f_{1,3}s_2}&\,\,_{f_{2,3}s_1s_2}&\,\,_{f_{1,2}s_2s_1}
        \end{array}
	$\fi
}

\newcommand{\figone}{
	\centering
	\ifnum\aspdf=1$
        \begin{array}{ccccc}
            \includegraphics{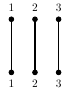}& 
            \includegraphics{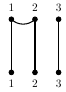}& 
            \includegraphics{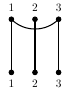}& 
            \includegraphics{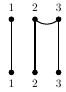}& 
            \includegraphics{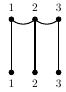}\\[-0.1cm] 
            \,\,_1&\,\,_{f_{1,2}}&\,\,_{f_{1,3}}&\,\,_{f_{2,3}}&\,\,_{f_{1,2}\,f_{2,3}}
        \end{array}
	$\else$
        \begin{array}{ccccc}
            \vpartition[type=\type]{{1,-1},{2,-2},{3,-3}}& 
            \vpartition[type=\type]{{-1,1,2,-2},{3,-3}}& 
            \vpartition[type=\type]{{2,-2},{-1,1,3,-3}}& 
            \vpartition[type=\type]{{1,-1},{-2,2,3,-3}}& 
            \vpartition[type=\type]{{-1,1,2,3,-3},{2,-2}}\\[-0.1cm] 
            \,\,_1&\,\,_{f_{1,2}}&\,\,_{f_{1,3}}&\,\,_{f_{2,3}}&\,\,_{f_{1,2}\,f_{2,3}}
        \end{array}
	$\fi\\[0.5cm]
}

\newcommand{\figzer}{
	\centering
	\ifnum\aspdf=1$
        \begin{array}{cccccc}
            \includegraphics{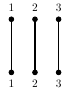}& 
            \includegraphics{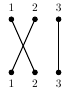}& 
            \includegraphics{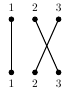}& 
            \includegraphics{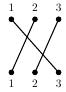}& 
            \includegraphics{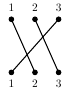}& 
            \includegraphics{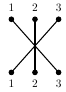}\\[-0.1cm] 
            \,\,_1&\,\,_{s_1}&\,\,_{s_2}&\,\,_{s_1s_2}&\,\,_{s_2s_1}&\,\,_{s_1s_2s_1}
        \end{array}
	$\else$
        \begin{array}{cccccc}
            \vpartition[type=\type]{{1,-1},{2,-2},{3,-3}}& 
            \vpartition[type=\type]{{1,-2},{2,-1},{3,-3}}& 
            \vpartition[type=\type]{{1,-1},{2,-3},{3,-2}}& 
            \vpartition[type=\type]{{1,-3},{2,-1},{3,-2}}& 
            \vpartition[type=\type]{{1,-2},{2,-3},{3,-1}}& 
            \vpartition[type=\type]{{1,-3},{2,-2},{3,-1}}\\[-0.1cm] 
            \,\,_1&\,\,_{s_1}&\,\,_{s_2}&\,\,_{s_1s_2}&\,\,_{s_2s_1}&\,\,_{s_1s_2s_1}
        \end{array}
	$\fi\\[0.5cm]
}